\documentclass{siamart0516}
\usepackage[utf8]{inputenc}
\usepackage{amsmath}
\usepackage{algorithmic}
\usepackage{amsfonts}
\usepackage{amssymb}
\newtheorem{thrm}{Theorem}
\newtheorem{mydef}{Definition}
\newtheorem{prop}{Proposition}
\newtheorem{rmk}{Remark}
\newtheorem{hyp}{Hypothesis}
\newtheorem{cor}{Corollary}
\usepackage{graphicx}
\usepackage{enumerate}
\usepackage{leftidx}
\usepackage{mathtools}

\def\({\left(}
\def\){\right)}
\def\[{\left[}
\def\]{\right]}
\def\r|{\right\rVert}
\def\l|{\left\lVert}

\def\0{{\mathbf 0}}
\def\bx{{\mathbf x}}
\def\by{{\mathbf y}}
\def\bz{{\mathbf z}}
\def\bw{{\mathbf w}}

\def\bv{{\mathbf v}}

\def\tr{{\mathbf{tr}}}
\def\sp{{\text{span}}}

\def\bA{\mathbf A}
\def\bB{\mathbf B}
\def\bC{\mathbf C}
\def\bD{\mathbf D}

\def\be{\mathbf e}
\def\bF{\mathbf F}
\def\bL{\mathbf L}

\def\bR{\mathbf R}
\def\bP{\mathbf P}
\def\bM{\mathbf M}
\def\bN{\mathbf N}
\def\bH{\mathbf H}
\def\bI{\mathbf I}
\def\bQ{\mathbf Q}

\def\bT{\mathbf T}

\def\bU{\mathbf U}
\def\bV{\mathbf V}
\def\bX{\mathbf X}

\def\bW{\mathbf W}

\def\bOmega{{\boldsymbol \Omega}}

\def\bPhi{{\boldsymbol \Phi}}

\def\bSigma{{\boldsymbol \Sigma}}
\def\bXi{{\boldsymbol \Xi}}

\def\bUpsilon{{\boldsymbol \Upsilon}}

\def\b1{{\mathbf 1}}

\def\T{{\mathrm T}}
\def\H{{\mathrm H}}

\def\O{{\mathcal O}}
\def\E{{\mathcal E}}

\newcommand{\TheTitle}{Asymptotic forecast uncertainty and the unstable subspace in the presence of additive model error}

\title{{\TheTitle}\thanks{Submitted to the editors July 26, 2017.
\funding{This work benefited from funding by the project REDDA of the Norwegian Research Council under contract 250711}}}

\author{
  Colin Grudzien\thanks{Nansen Environmental and Remote Sensing Center, Bergen, Norway
    (\email{colin.grudzien@nersc.no}, \url{https://cgrudz.github.io}).}
  \and
  Alberto Carrassi\thanks{Nansen Environmental and Remote Sensing Center, Bergen, Norway (\email{alberto.carrassi.nersc.no}})
  \and
  Marc Bocquet,\thanks{CEREA, joint laboratory \'{E}cole des Ponts ParisTech and EDF R\&D, Universit\'{e} Paris-Est,
Champs-sur-Marne, France (\email{marc.bocquet@enpc.fr})
}
}
\begin{document}

\maketitle

% REQUIRED
\begin{abstract}
It is well understood that dynamic instability is among the primary drivers of forecast uncertainty in chaotic, physical systems.  Data assimilation techniques have been designed to exploit this phenomena, reducing the effective dimension of the data assimilation problem to the directions of rapidly growing errors.  Recent mathematical work has, moreover, provided formal proofs of the central hypothesis of the Assimilation in the Unstable Subspace methodology of Anna Trevisan and her collaborators: for filters and smoothers in perfect, linear, Gaussian models, the distribution of forecast errors asymptotically conforms to the unstable-neutral subspace.  Specifically, the column span of the forecast and posterior error covariances asymptotically align with the span of backward Lyapunov vectors with non-negative exponents.  

Earlier mathematical studies have focused on perfect models, and this current work now explores the relationship between dynamical instability, the precision of observations and the evolution of forecast error in linear models with additive model error.  We prove bounds for the asymptotic uncertainty, explicitly relating the rate of dynamical expansion, model precision and observational accuracy.  Formalizing this relationship, we provide a novel, necessary criterion for the boundedness of forecast errors.  Furthermore, we numerically explore the relationship between observational design, dynamical instability and filter boundedness.  Additionally, we include a detailed introduction to the Multiplicative Ergodic Theorem and to the theory and construction of Lyapunov vectors. 

While forecast error in the stable subspace may not generically vanish, we show that even without filtering, uncertainty remains uniformly bounded due its dynamical dissipation.  
However, the continuous re-injection of uncertainty from model errors may be excited by transient instabilities in the stable modes of high variance, rendering forecast uncertainty impractically large.  In the context of ensemble data assimilation, this requires rectifying the rank of the ensemble-based gain to account for the growth of uncertainty beyond the unstable and neutral subspace, additionally correcting stable modes with frequent occurrences of positive local Lyapunov exponents that excite model errors.

\end{abstract}

% REQUIRED
\begin{keywords}
 Kalman filter, data assimilation, model error, Lyapunov vectors, control theory
\end{keywords}

% REQUIRED
\begin{AMS}
  93E11, 93C05, 93B07, 60G35, 15A03
\end{AMS}

\section{Introduction}

The seminal work of Lorenz \cite{lorenz63} demonstrated that, even in deterministic systems, infinitesimal perturbations in initial conditions can rapidly lead to a long-term loss of predictability in chaotic, physical models.  In weather prediction, this understanding led to the transition from single-trajectory forecasts to ensemble-based, probabilistic forecasting \cite{leutbecher2008ensemble}.  Historically, ensembles have been initialized in order to capture the spread of rapidly growing perturbations \cite{buizza93,toth1997ensemble}.  Data assimilation methods have likewise been designed to capture this variability in the context of Bayesian and variational data assimilation schemes; see e.g., Carrassi et al. for a recent survey of data assimilation techniques in geosciences \cite{carrassi2018data}.  The ensemble Kalman filter, particularly, has been shown to strongly reflect these dynamical instabilities \cite{carrassi2009, ng2011role, gonzalez2013ensemble, bocquet2017four}, and its performance depends significantly upon whether these rapidly growing errors are sufficiently observed and corrected.   

The Assimilation in the Unstable Subspace (AUS) methodology of Trevisan et. al. \cite{carrassi2008b, trevisan2010, trevisan2011kalman, palatella2013a, palatella2015} has provided a robust, dynamical interpretation of these observed properties of the ensemble Kalman filter.  For deterministic, linear, Gaussian models, Trevisan et al. hypothesized that the asymptotic filter error concentrates in the span of the unstable-neutral backward Lyapunov vectors (BLVs), and this has recently been mathematically proven.  Gurumoorthy et. al. \cite{gurumoorthy2017} demonstrated that the null space of the forecast error covariance matrices asymptotically contain the time varying subspace spanned by the stable BLVs.  This result was generalized by Bocquet et. al. \cite{bocquet2017degenerate}, proving the asymptotic equivalence of reduced rank initializations of the Kalman filter with the full rank Kalman filter: as the number of assimilations increases towards infinity, the covariance of the full rank Kalman filter converges to a sequence of low rank covariance matrices initialized only in the unstable-neutral BLVs.

The convergence of the Kalman smoother error covariances onto the span of the unstable-neutral BLVs, and stability of low rank initializations, was established by Bocquet \& Carrassi \cite{bocquet2017four}; this latter work also numerically extended this relationship to weakly nonlinear dynamics and ensemble-variational methods.  The works of Bocquet et al. \cite{bocquet2017degenerate} and Bocquet \& Carrassi \cite{bocquet2017four} relied upon the \emph{sufficient} hypothesis that the span of the unstable and neutral BLVs remained uniformly-completely observed.  This hypothesis has recently been refined to a \emph{necessary and sufficient} criterion for the exponential stability of continuous time filters, in perfect models, in terms of the detectability of the unstable-neutral subspace \cite{frank2017detectability}.  

The present study is concerned with extending the limits of the results developed in deterministic dynamics (perfect models), now to the presence of stochastic model errors.  This manuscript and its sequel \cite{grudzien2018inflation} seek to: (i) determine the extent to which stable dynamics confine the uncertainty in the sequential state estimation problem in models with additive noise, and (ii) to use these results to interpret the properties, and suggest design, of ensemble-based Kalman filters with model error.  This manuscript studies the asymptotic properties of the full rank, theoretical Kalman filter \cite{kalman1960}, and the unfiltered errors in the stable BLVs.  The sequel \cite{grudzien2018inflation} utilizes these results to interpret filter divergence for reduced rank, ensemble-based Kalman filters.

In \cref{section:QR} we present a detailed introduction to the BLVs.  In section \cref{section:bounds} develop novel bounds on the forecast error covariance, describing the evolution of uncertainty as the growth of error, due to dynamic instability and model imprecision, with respect to the constraint of observations.  Together, the rate of dynamic instability and the observational precision form an inverse relationship which we use to characterize the boundedness of forecast errors.  In \cref{cor:lowerbndspread} and \cref{cor:tvnecessary}, we prove a necessary criterion for filter boundedness in autonomous and time varying systems: the observational precision, relative to the background uncertainty, must be greater than the leading instability which forces the model error.  Our results derive from the bounds provided in \cref{prop:autupper} for autonomous dynamics and \cref{prop:tvbnd} for time varying systems.  An important consequence is that under generic assumptions, forecast errors in the span of the stable BLVs remain uniformly bounded \emph{independently} of filtering.  Described in \cref{cor:stablebnd}, this extends the intuition of AUS now to the presence of model errors: filters need only target corrections to the span of the unstable and neutral BLVs to maintain bounded errors.  

However, the intuition of AUS needs additional qualifications when interpreting the role of model errors in reduced rank filters. Unlike perfect models, uncertainty in the stable  BLVs does not generically converge to zero as a consequence of re-introducing model errors.  Moreover, while stability guarantees that unfiltered errors remain uniformly bounded in the stable BLVs, the uncertainty may still be impractically large due: even when a Lyapunov exponent is strictly negative, positive realizations of the \emph{local} Lyapunov exponents can force transient instabilities which strongly amplify the forecast uncertainty.  The impact of stable modes on forecast uncertainty differs from similar results for nonlinear, perfect models by Ng. et. al. \cite{ng2011role}, and Bocquet et. al. \cite{bocquet2015expanding}, where the authors demonstrate the need to correct stable modes in the ensemble Kalman filter due to sampling errors induced by nonlinearity.  Likewise, this differs from the EKF-AUS-NL of Palatella \& Trevisan \cite{palatella2015}, that accounts for truncation errors in the estimate of the forecast uncertainty in nonlinear models.  In \cref{section:localvariability}, we derive the mechanism for the transient instabilities amplifying perturbations as a linear effect in the presence of model errors.  We furthermore provide a computational framework to study the variance of these perturbations.

In \cref{section:kferr}, we study the filter boundedness and stability criteria of Bocquet et al. \cite{bocquet2017degenerate} and Frank \& Zhuk \cite{frank2017detectability} in their relation to bounding forecast errors in imperfect models.  Likewise, we explore their differences in the context of dynamically selecting observations, similar to the work of Law et al. \cite{law2016filter}.  With respect to several observational designs as benchmarks, we numerically demonstrate that the \emph{unconstrained growth} of errors in the stable BLVs of high variance can be impractically large compared to the uncertainty of the full rank Kalman filter.  
These results have strong implications for ensemble-based filtering in geosciences and weather prediction, where ensemble sizes are typically extremely small relative to the model dimension.  In perfect models, an ensemble size large to correct the small number unstable and neutral modes might suffice.  However, our results suggest the need to further increase the rank of ensemble-based gains.  The significance of this result for ensemble-based Kalman filters and their divergence is further elaborated on in the sequel \cite{grudzien2018inflation}.

\section{Linear state estimation}

The purpose of recursive data assimilation is estimating an unknown state with a sequential flow of partial and noisy observations; we make the simplifying assumption that the dynamical
and observational models are both linear and the error distributions are Gaussian.  In this setting, given a Gaussian distribution for the initial state, the distribution of the estimated state is Gaussian at all times.  Formulated as a Bayesian inference problem, we seek to estimate the distribution of the random vector $\bx_k \in \mathbb{R}^n$ evolved via a linear Markov model,
\begin{align}\label{eq:dynmodel}
 \bx_{k} &= \bM_{k}\bx_{k-1} + \bw_k ,
\end{align} 
with observations $\by_k \in \mathbb{R}^d$ given as
\begin{align}
\by_{k} &= \bH_k \bx_k + \bv_k.  \label{eq:obsmodel}
 \end{align}
The model variables and observation vectors are related via
the linear observation operator $\bH_k: {\mathbb R}^n \mapsto {\mathbb R}^d$.  Let $\bI_{n}$ denote the $n \times n$ identity matrix.  We denote the model propagator from time $t_{l-1}$ to time $t_{k}$ as $\bM_{k:l} \triangleq \bM_k \cdots \bM_l$, where $\bM_{k:k} \triangleq \bI_n$.

For all $k, l \in \mathbb{N}$, the random vectors of model and observation noise, $\bw_k, \bw_l\in \mathbb{R}^n$ and $\bv_k,\bv_l\in\mathbb{R}^d$, are assumed mutually independent, unbiased, Gaussian white
sequences.  Particularly, we define
\begin{equation}
   \mathbb{E}[\bv_k\bv_l^\T] = \delta_{k,l}\bR_k \quad \text{ and} \quad  \mathbb{E}[\bw_k\bw_l^\T] = \delta_{k,l}\bQ_k,
\end{equation}
where $\mathbb{E}$ is the expectation, $\bR_k\in {\mathbb R}^{d\times d}$ is the observation error covariance matrix at time $t_k$, and $\bQ_k \in
{\mathbb R}^{n\times n}$ stands for the model error covariance matrix. The error covariance matrix $\bR_k$ can be
assumed invertible without losing generality.  For simplicity we assume the dimension of the observations $d\leq n$ will be fixed.  

For two positive semi-definite matrices, $\bA$ and $\bB$, the partial ordering is defined
$\bB \leq \bA$ if and only if $\bA- \bB$ is positive semi-definite. To avoid pathologies, we assume that the model error and the observational error covariance matrices are uniformly bounded, i.e., there are constants $q_{\inf}, q_{\sup}, r_{\inf}, r_{\sup} \in \mathbb{R}$ such that for all $k$,
\begin{align}
\0 &\leq q_{\inf}\bI_n \leq \bQ_k \leq q_{\sup} \bI_n,\\
\0 &< r_{\inf} \bI_d \leq \bR_k \leq r_{\sup} \bI_d.
\end{align}

Rather than explicitly computing the evolution of the distribution for $\bx_k$, the Kalman filter computes the forecast and posterior distributions parametrically via recursive equations for the mean and covariance of each distribution. 
\begin{mydef}
The forecast error covariance matrix $\bP_k$ of the Kalman filter satisfies the discrete-time
dynamic Riccati equation \emph{\cite{kalman1960}}
\begin{align}
\label{eq:recurrence}
  \bP_{k+1} = \bM_{k+1}\(\bI_n+\bP_k \bOmega_k\)^{-1}\bP_k\bM_{k+1}^{\T} +\bQ_{k+1} ,
\end{align}
where $\bOmega_k \triangleq \bH_k^{\T}\bR_k^{-1}\bH_k$ is the \textbf{precision matrix} of the observations.  
\end{mydef}

Equation \cref{eq:recurrence} expresses the error covariance matrix, $\bP_{k+1}$, as the result of a two-step
process: (i) the \emph{assimilation} at time $t_k$ yielding the analysis error covariance,
\begin{align}
\label{eq:kfupdate} 
\bP^{a}_k = \(\bI_n+\bP_k \bOmega_k\)^{-1}\bP_k;
\end{align}
and (ii) the \emph{forecast}, where the analysis error covariance is forward propagated by
\begin{align}
\label{eq:kfforecast} 
\bP_{k+1} = \bM_{k+1}\bP^{ a}_k\bM_{k+1}^{\T} + \bQ_{k+1} .
\end{align}
Assuming that the filter is unbiased, such that the initial error is mean zero, it is easy to demonstrate that the forecast and analysis error distributions are mean zero at all times.  In this context, the covariances $\bP_k,\bP^a_k$ represent the uncertainty of the state estimate defined by the filter mean.  As we will focus on the evolution of the covariances, we neglect the update equations for the mean state and refer the reader to Jazwinski \cite{jazwinski1970} for a more complete discussion.

The classical conditions for the 
boundedness of filter errors, and the independence of the asymptotic filter behavior from its initialization, are given in terms of observability and controllability.  Observability is the condition that given finitely many observations, the initial state of the system can be reconstructed.
Controllability describes the ability to move the system from any
initial state to a desired state given a finite sequence of control actions --- in our case the moves are the realizations of model error.  These conditions are described in the following definitions, beginning with the information and controllability matrices.
\begin{mydef}\label{def:obscontmat}
We define $\bPhi_{k:j}$ to be the time varying \textbf{information matrix} and $\bUpsilon_{k:j}$ to be the time varying \textbf{controllability matrix}, where
\begin{align}
\bPhi_{k:j} &\triangleq \sum_{l=j}^{k} \bM_{k:l}^{-\T} \bOmega_l \bM_{k:l}^{-1}, &  &
\bUpsilon_{k:j}\triangleq \sum_{l=j}^{k} \bM_{k:l} \bQ_l \bM_{k:l}^\T.
\end{align}
For $\gamma\geq 0$ let us define the \textbf{weighted controllability matrix} as
\begin{align}
\bXi^\gamma_{k:j} \triangleq \sum_{l=j}^k \(\frac{1}{1+\gamma}\)^{k-l} \bM_{k:l} \bQ_l \bM_{k:l}^\T.
\end{align}
\end{mydef}
Note that, $\bXi^0_{k:j} \equiv \bUpsilon_{k:j}$.  We recall from section 7.5 of Jazwinski \cite{jazwinski1970} the definitions of uniform complete observability (respectively controllability).
\begin{mydef}\label{def:obscont}
Suppose there exists $N_\Phi, a,b > 0$ independent of $k$ such that $k> N_\Phi$ implies
\begin{align}
0< a \bI_n \leq \bPhi_{k:k - N_\Phi} \leq b \bI_n,
\end{align}
then the system is \textbf{uniformly completely observable}.  Likewise suppose there exists $N_{\Upsilon}, a,b >0$ independent of $k$ for which $k> N_\Upsilon$ implies
\begin{align}
0< a \bI_n \leq \bUpsilon_{k:k- N_\Upsilon} \leq b \bI_n, \label{eq:controlbound}
\end{align}
then the system is \textbf{uniformly completely controllable}.
\end{mydef}

\begin{hyp}\label{hyp:obscont}
Assume that the system of equations \cref{eq:dynmodel} and \cref{eq:obsmodel} is 
\begin{enumerate}[(a)]
\item uniformly completely observable; 
\item uniformly completely controllable.
\end{enumerate}  
\end{hyp}
\begin{rmk}
We will explicitly refer to \cref{hyp:obscont} whenever it is used.  When we refer \cref{hyp:obscont} alone, we refer to both parts $(a)$ and $(b)$.  At times, we will explicitly only use either part (a) or (b) of \cref{hyp:obscont}.
\end{rmk}
\begin{thrm}\label{thrm:boundedstable}
Suppose the system of equations \cref{eq:dynmodel} and \cref{eq:obsmodel} satisfies \cref{hyp:obscont} and $\bP_0 >0$.  Then there exists constants $p^a_{\inf}$ and $p^a_{\sup}$ independent of $k$ such that the analysis error covariance is uniformly bounded above and below,
\begin{align}
0 < p^a_{\inf} \bI_n \leq \bP^a_k \leq p^a_{\sup} \bI_n < \infty.
\end{align} 
Given any two initializations of the prior error covariance $\bP_0,\widehat{\bP}_0>0$, with associated sequences of analysis error covariances $\bP^a_k,\widehat{\bP}^a_k$, the covariance sequences converge, $\lim_{k\rightarrow \infty} \l|\bP^a_k - \widehat{\bP}^a_k \r| = 0 $,  exponentially in $k$.
\end{thrm}
These are classical results of filter stability, see for example Theorem 7.4 of Jazwinski \cite{jazwinski1970}, or Bougerol's work with random matrices \cite{bougerol1993, bougerol1995} for a generalization.    

The square root Kalman filter is a reformulation of the recurrence in equation \cref{eq:recurrence} which is used to reduce computational cost and obtain superior numerical precision and stability over the standard implementations see, e.g., \cite[and references therein]{tippett2003}.  The advantage of this formulation to be used in our analysis is to explicitly represent the recurrence in equation \cref{eq:recurrence} in terms of positive semi-definite, symmetric matrices.

\begin{mydef}\label{def:squareroot}
Let $\bP_k$ be a solution to the time varying Riccati equation \cref{eq:recurrence} and define $\bX_k \in \mathbb{R}^{n\times n}$ to be a Cholesky factor of $\bP_k$, such that
\begin{align}\label{eq:cholesky}
\bP_k = \bX_k \bX^\T_k.
\end{align}  
\end{mydef}

The root $\bX_k$ in equation \cref{eq:cholesky} can be interpreted as an ensemble of anomalies about the mean as in the ensemble Kalman filter \cite{evensen2009data,asch2016data}.  In operational conditions, it is standard that the forecast error distribution is approximated with a sub-optimal, reduced rank surrogate \cite{chandrasekar2008}.  Using a reduced rank approximation, the estimated covariance and exact error covariance are not equal, and this can lead to the systematic underestimation of the uncertainty \cite{grudzien2018inflation}.  However, in the following we will assume that $\bX_k$ is computed as an exact root.  The sequel to this work explicitly treats the case of reduced rank, sub-optimal filters \cite{grudzien2018inflation}.

\begin{mydef}\label{def:alphabeta}
We order singular values $\sigma_1 > \cdots > \sigma_n$ such that,
\begin{align}\label{eq:sigmabound}
&0 \leq \sigma_n\( \bX_k^\T \bOmega_k \bX_k \) \bI_n \leq \bX_k^\T \bOmega_k \bX_k \leq  \sigma_1\(\bX_k^\T \bOmega_k \bX_k\) \bI_n <\infty.
\end{align}
We define
\begin{align}
&\alpha\triangleq \inf_k \left\{\sigma_n\(\bX_k^\T \bOmega_k \bX_k\)\right\} \geq 0, & 
&\beta\triangleq \sup_k \left\{\sigma_1\(\bX_k^\T \bOmega_k \bX_k\)\right\} \leq \infty,
\end{align}
and we write $0\leq \alpha \bI_n \leq \bX_k^\T \bOmega_k \bX_k \leq \beta \bI_n \leq \infty$ for all $k$.
\end{mydef}

Equation \cref{eq:sigmabound} is closely related to the singular value analysis of the precision matrix by Johnson et al. \cite{johnson2005singular} and the analysis of the conditioning number for the Hessian of the variational cost function by Haben et al. \cite{haben2011conditioning} and Tabeart et al. \cite{tabeart2018conditioning}.  These works study the information gain from observations, relative to the background uncertainty, due to the assimilation step.  The primary difference between these earlier works and our study here is that the background error covariance is static in these variational formulations, while in the present study the root $\bX_k$ is flow dependent.  In this flow dependent context, the constant $\alpha$ (respectively $\beta$) is interpreted as the minimal (respectively maximal) observational precision relative to the maximal (respectively minimal) background forecast uncertainty.  The constant $\alpha$ is nonzero if and only if the principal angles between the column span of $\bX_k$ and the kernel of $\bH_k$ are bounded uniformly below.  Generally, we thus take $\alpha=0$ unless observations are full dimensional.  A nonzero value for $\alpha$ can be understood as an ideal scenario.

Using \cref{def:squareroot} and the matrix shift lemma \cite[see Appendix C]{bocquet2017degenerate} we re-write the forecast Riccati equation \cref{eq:recurrence} as 
\begin{align}
  \bP_k &= \bM_k(\bI_n+ \bP_{k-1}\bOmega_{k-1})^{-1}\bP_{k-1}\bM_k^\T + \bQ_k \\
&= \bM_k \bX_{k-1}(\bI_n + \bX^\T_{k-1} \bOmega_{k-1} \bX_{k-1})^{-1}\bX^\T_{k-1}\bM_k^\T + \bQ_k \label{eq:riccatitv}
\end{align}
from which we infer
\begin{align}
\frac{1}{1+ \beta} \bM_k \bP_{k-1}\bM_k^\T + \bQ_k \leq \bP_k \leq \frac{1}{1+ \alpha} \bM_k \bP_{k-1}\bM_k^\T + \bQ_k.
\end{align}
Iterating on the above inequality, we obtain the recursive bound
\begin{align}
\(\frac{1}{1+ \beta}\)^k \bM_{k:0} \bP_{0}\bM_{k:0}^\T + \bXi^{\beta}_{k:1} \leq \bP_k \leq \(\frac{1}{1+ \alpha}\)^k \bM_{k:0} \bP_{0}\bM_{k:0}^\T + \bXi^{\alpha}_{k:1}.\label{eq:nonautrecursion}
\end{align}
\begin{rmk}
Equation \cref{eq:nonautrecursion} holds if there is no filtering step, setting $\beta = \alpha =0 $.
\end{rmk} 

The bounds in equation \cref{eq:nonautrecursion} explicitly describe the previously introduced uncertainty as dynamically evolved to time $k$, relative to the constraint of the observations.  We will utilize the BLVs vectors to extract the dynamic information from the sequences of matrices $\bM_{k:l}, \bM_{k:l}^\T$.

\section{Lyapunov vectors}
\label{section:QR}

This section contains a short introduction to Lyapunov vectors and the Multiplicative Ergodic Theorem (MET).  For a more comprehensive introduction, there are many excellent resources at different levels of complexity, see for example \cite{adrianova1995introduction, legras1996, barreira2002, barreira2013, Kuptsov2012, froyland2013computing}.  There is inconsistent use of the terminology for Lyapunov vectors in the literature, so we choose to use the nomenclature of Kuptsov \& Parlitz \cite{Kuptsov2012} for its accessibility and self-consistency.

Consider the growth or decay of an arbitrary, norm one vector $v_0 \in \mathbb{R}^n$ to its state at time $t_k$ via the propagator $\bM_{k:0}$.  This is written as 
\begin{align}
\l| v_k \r| =  \l| \bM_{k:0} v_0 \r| =  \sqrt{   v_0^\T \bM_{k:0}^\T \bM_{k:0} v_0  },\label{eq:forwardgrowth}
\end{align}
so that the eigenvectors of the matrix $\bM_{k:0}^\T \bM_{k:0}$ describe the principal axes of the ellipsoid defined by the unit disk evolved to time $t_k$.  Using the above relationship for the reverse time model, we see the growth or decay of the unit disk in reverse time as
\begin{align}
\l| u_{-k} \r| =  \l| \bM^{-1}_{0:-k} u_0 \r| = \sqrt{   u_0^\T \bM_{0:-k}^{-\T} \bM^{-1}_{0:-k} u_0  }. \label{eq:backwardgrowth}
\end{align}
The principal axes of the past ellipsoid that evolves to the unit disk at the present time are thus precisely the eigenvectors of the matrix $\bM^{-\T}_{0:-k}\bM^{-1}_{0:-k}$.
There is no guarantee in general that there is consistency between the asymptotic forward and reverse time growth and decay rates, i.e., in equations 
\cref{eq:forwardgrowth} and \cref{eq:backwardgrowth} as $k \rightarrow \infty$.  Generally, models may have Lyapunov spectrum defined as intervals of lower and upper growth rates, see e.g., Dieci \& Van Vleck \cite{dieci2002lyapunov, dieci2007lyapunov}.  However, as we are motivated by the tangent-linear model for a nonlinear system, we may assume some ``regularity'' in the dynamics.

The anti-symmetry of the forward/reverse time, regular and adjoint models' growth and decay is known as Lyapunov-Perron regularity (\textbf{LP-regularity}) \cite{barreira2013}, and is equivalent to the classical Oseledec decomposition \cite{barreira2002}[see Theorem 2.1.1].  LP-regularity guarantees that: (i) the Lyapunov exponents are well defined for the linear model as point-spectrum; (ii) the linear space is decomposable into subspaces that evolve covariantly with the linear propagator; and (iii) each such subspace asymptotically grows or decays according to one of the point-spectrum rates.  We summarize the essential results of Oseledec's theorem for use in our work in the following, see Theorem 2.1.1 of Barreira \& Pesin \cite{barreira2002} for a complete statement.

\begin{thrm}[Oseledec's Theorem]
The model $\bx_k = \bM_k \bx_{k-1}$ is LP-regular if and only if there exists real numbers $\lambda_1 > \cdots> \lambda_p$, for $1 \leq p \leq n$, and subspaces $\E_k^i \subset \mathbb{R}^n$, $\dim\(\E^i_k\) = \kappa_i$, such that for every $k,l \in \mathbb{Z}$
\begin{equation}\begin{matrix}
\bigoplus_{i=1}^p \E^i_k = \mathbb{R}^n & & \bM_{k \pm l:k } \E^i_k = \E^i_{k \pm l},
\end{matrix}\label{eq:oseledecsplitting}\end{equation}
and $\bv\in\E^i_k$ implies
    \begin{align}    \label{eq:covariantgrowth}
    \lim_{l\rightarrow \infty} \frac{1}{ l } \log(\l| \bM_{k \pm l:k} \bv \r|) = \pm \lambda_i .
    \end{align}
\end{thrm}
\begin{mydef}\label{def:lyapexps}
For $p\leq n$, the \textbf{Lyapunov spectrum} of the system \cref{eq:dynmodel} is defined as the set 
$\{\lambda_i : \kappa_i \}_{i=1}^p$
where $\lambda_1 > \cdots > \lambda_p$ and $\kappa_i$ corresponds to the multiplicity (degeneracy) of the exponent $\lambda_i$.
We separate \emph{non-negative} and \emph{negative} exponents, $\lambda_{n_0} \geq 0 > \lambda_{n_0 +1}$, such that each index $i> n_0$ corresponds to a stable exponent.  The subspaces $\E^i_k$ are denoted \textbf{Oseledec spaces}, and the decomposition of the model space into the direct sum is denoted Oseledec splitting.
\end{mydef}

 For arbitrary linear systems LP-regularity is not a generic property --- it is the MET that shows that this is a typical scenario for a wide class of nonlinear systems.  A point will be defined to be LP-regular if the tangent-linear model along its evolution is LP-regular.  We state a classical version of the MET \cite[see Theorem 2.1.2 and the following discussion]{barreira2002} but note that there are more general formulations of this result and more general forms of the associated covariant-subspace decompositions.  These results go beyond the current work, see e.g., Froyland et al. \cite{froyland2013computing} and Dieci et al. \cite{dieci2007lyapunov,dieci2010exponential} for a stronger version of the MET and related topics.

\begin{thrm}[Multiplicative Ergodic Theorem]\label{thrm:MET}
If $f$ is a $C^1$ diffeomorphism of a compact, smooth, Riemannian manifold $M$, the set of points in $M$ which are LP-regular has measure 1 with respect to any $f$-invariant Borel probability measure $\nu$ on $M$.  If $\nu$ is ergodic, then the Lyapunov spectrum is constant with $\nu$-probability 1.
\end{thrm}
Loosely, the MET states that, with respect to an ergodic probability measure (that is compatible with the map $f$ and the usual topology), there is probability one of choosing initial conditions for which the Lyapunov exponents are well defined and independent of initial condition.  This form of the MET has a wide range of applications in differentiable dynamical systems, but the MET is not limited to this setting.  The strong version of the MET has been applied in, e.g., hard disk systems, the truncated Fourier expansions of PDEs, and with non-autonomous ODEs and their transfer operators \cite[See example 1.2 for a discussion of these topics]{froyland2013computing}.  For the rest of this work, we will take the hypothesis that our model satisfies LP-regularity.
\begin{hyp}\label{hyp:lpregular}
The model defined by the deterministic equation 
\begin{align}
\bx_k = \bM_k \bx_{k-1},\label{eq:deterministic}
\end{align}
is assumed to be LP-regular.
\end{hyp}
The deterministic evolution in equation \cref{eq:deterministic} comes naturally in the formulation of the Kalman filter, where the mean state is evolved via the deterministic component in the forecast step.  For Gaussian error distributions, the evolution of the forecast error distribution is interpreted in terms of Oseledec's theorem as the evolution of deviations from the mean, propagated via the equations for perturbations.  While Oseledec's theorem guarantees that a decomposition of the model space exists, constructing such a decomposition is non-trivial.  Motivated by equations \cref{eq:forwardgrowth} and \cref{eq:backwardgrowth}, we define the following operators as in equations (13) and (14) of Kuptsov \& Parlitz \cite{Kuptsov2012}.
\begin{mydef}\label{def:faroperators}
We define the \textbf{far-future operator} as
\begin{align}\label{eq:farfuture}
\bW^+(k) \triangleq & \lim_{l\rightarrow \infty}\[\bM^{\T}_{k+l: k} \bM_{ k+l: k}\]^{\frac{1}{2l}},
\end{align}
and the \textbf{far-past operator} as
\begin{align}\label{eq:farpast}
\bW^-(k) \triangleq & \lim_{l\rightarrow \infty}\[\bM^{-\T}_{k: k-l} \bM^{-1}_{ k: k-l}\]^{\frac{1}{2l}}.
\end{align}
\end{mydef}

In the classical proof of the MET, the far-future/past operators are shown to be well defined positive definite, symmetric operators \cite{oseledec1968multiplicative}.  As they are diagonalizable over $\mathbb{R}$, we order the eigenvalues of $\bW^+(k)$ as $\mu^+_1(k)>\cdots >\mu^+_p(k)$ and the eigenvalues of $\bW^-(k)$ as $\mu^-_p(k) > \cdots > \mu^-_1(k)$.  
By the MET, the eigenvalues $\mu^\pm_i(k)$ are independent of $k$ and satisfy the relationship
\begin{align}
\log(\mu^\pm_i) = \pm \lambda_i.\label{eq:eigenvalue_farlimit}
\end{align}

\begin{mydef}\label{def:lyapvects}
Let the columns of the matrix $\bF_k$, respectively $\bB_k$, be any orthonormal eigenbasis for the far-future operator $\bW^+(k)$, respectively far past operator $\bW^-(k)$.  Order the columns block-wise, such that for each $i =1,\cdots, p$ and each $j =1,\cdots, \kappa_i$, $\bF^{i_j}_k$ is an eigenvector for $\mu^+_i$ and $\bB^{i_j}_k$ is an eigenvector for $\mu^-_i$.  We define $\bF^{i_j}_k$ to be the $i_j$-th \textbf{forward Lyapunov vector (FLV)} at time $k$ and $\bB_k^{i_j}$ to be the $i_j$-th \textbf{backward Lyapunov vector (BLV)}. Let the columns of $\bC_k$ form any basis such that for each $i=1,\cdots, p$ and each $j=1,\cdots, \kappa_i$, $\bC^{i_j}_k \in \E^i_k$.  Then we define $\bC^{i_j}_k$ to be the $i$-th \textbf{covariant Lyapunov vector (CLV)} at time $k$.
\end{mydef}

The CLVs are defined only by the Oseledec spaces, and therefore, are independent of the choice of a norm --- any choice of basis subordinate to the Oseledec splitting is valid.  On the other hand, the FLVs and the BLVs are determined specifically with respect to a choice of a norm and the induced metric.  The choice of basis in each case can be made uniquely (up to a scalar and the choice of a norm) only when $p=n$.  For the remaining work we will focus on the BLVs; for a general survey on constructing FLVs, BLVs and CLVs, see e.g., Kuptsov \& Parlitz \cite{Kuptsov2012} and Froyland et. al. \cite{froyland2013computing}.

The Oseledec spaces and Lyapunov vectors can also be defined in terms of filtrations, i.e., chains of ascending or descending subspaces of $\mathbb{R}^n$.  This forms an axiomatic approach to constructing abstract Lyapunov exponents used by, e.g., Barreira \& Pesin \cite{barreira2002}.  The BLVs describe an orthonormal basis for the ascending chain of Oseledec subspaces, the \emph{backward filtration} \cite{Kuptsov2012}.  For all $1 \leq m\leq p$ we obtain the equality
\begin{align}\label{eq:filtrationbasis}\bigoplus_{i=1}^m \E^i_k = \bigoplus_{i=1}^m \sp\left\{\bB_k^{i_j}\right\}_{j=1}^{\kappa_i},
\end{align}
by equation (17) by Kuptsov \& Parlitz \cite{Kuptsov2012}, and the decomposition of the backward filtration in equations (1.5.1) and (1.5.2) of Barreira \& Pesin \cite{barreira2002}.  Note that equation \cref{eq:filtrationbasis} \emph{does not} imply $\bB^{i_j}_k \in \E^i_k$ for $i> 1$, as the BLVs are not themselves covariant with the model dynamics.  However, the BLVs are covariant with the QR algorithm.     
\begin{lemma}
Outside of a set of Lebesgue measure zero, a choice of $i_j$ linearly independent initial conditions for the recursive QR algorithm converges to some choice for the leading $i_j$ BLVs.  For any $k$, the BLVs satisfy the relationship 
\begin{align}\label{eq:qrrecursion}
\bM_k \bB_{k-1} =  \bB_k \bT_k, \hspace{2mm} \Leftrightarrow \hspace{2mm} \bM_k = \bB_k \bT_k \bB_{k-1}^\T
\end{align}
where $\bT_k$ is an upper triangular matrix. Moreover, for any $i_j$ and any $k$,
\begin{align}
\lim_{l \rightarrow -\infty} \frac{1}{k-l} \log\( \l|\bM^{\T}_{k:l} \bB^{ i_j}_{k} \r|\) = \lambda_i. \label{eq:backwardsgrowth}
\end{align}
\end{lemma}
\begin{proof}
The covariance of the BLVs with respect to the QR algorithm in equation \cref{eq:qrrecursion} can be derived from equations \cref{eq:oseledecsplitting} and \cref{eq:filtrationbasis}.  For all $1\leq m \leq p$,
\begin{align}
\bM_k \(\bigoplus_{i=1}^m \sp\left\{\bB_{k-1}^{i_j}\right\}_{j=1}^{\kappa_i}\) = \bigoplus_{i=1}^m \sp\left\{\bB_{k}^{i_j}\right\}_{j=1}^{\kappa_i},
\end{align}
due to the covariance of the Oseledec spaces.  Therefore the transformation $\bM_k$ represented in a moving frame of BLVs is upper triangular.  When the spectrum is degenerate, $p<n$, there is non-uniqueness in the choice of the BLVs.  However, given an initial choice of the BLVs at some time $k-1$, the choice of BLVs at time $k$ can be defined directly via the relationship in \cref{eq:qrrecursion}.  This is the relationship derived in equation (31) by Kuptsov \& Parlitz \cite{Kuptsov2012}, and is the basis of the recursive QR algorithms of Shimada \& Nagashima \cite{shimada1979} and Benettin et. al. \cite{benettin1980}.  A choice of BLVs gives a special choice of the classical Perron transformation \cite{adrianova1995introduction}[see Theorems 3.3.1 \& 3.3.2], and in particular, it is proven by Ershov \& Potapov \cite{ershov1998} that outside of a set of Lebesgue measure zero, the recursive QR algorithm converges to some choice of BLVs.

Note that the far-future/past operators are  also well defined for the propagator of the adjoint model $\bz_k = \bM^{-\T}_k \bz_{k-1}.$
Equation \cref{eq:backwardsgrowth} thus follows from the far-past operator for the adjoint model, defined
\begin{align}\label{eq:adjointfarpast}
\bW^{\ast -}(k) \triangleq \lim_{l \rightarrow \infty} \[ \bM_{k:k-l} \bM_{k:k-l}^\T\]^{\frac{1}{2l}}.
\end{align}
It is easy to verify that the BLVs defined by the adjoint model agree with those defined via the regular model --- in each case, the left singular vectors of $\bM_{k:k-l}$ converge to a choice of the BLVs as $l\rightarrow \infty$.  Notice that the eigenvalues of $\bW^{\ast -}(k)$ are reciprocal to those of $\bW^-(k)$, i.e.,
$\mu^{\ast -}_i = \frac{1}{\mu^-_i}$.  Thus by equation \cref{eq:eigenvalue_farlimit}, $\log\(\mu^{\ast -}_i \) =  \lambda_i$.
\end{proof}

Equation \cref{eq:qrrecursion} describes the dynamics in the moving frame of BLVs, where the transition map from the frame at time $t_{k-1}$ to time $t_k$ is given by $\bT_k$.  Applying the change of basis sequentially for the matrix $\bM_{k:l}$, we recover
\begin{equation}
 \bM_k =  \bB_k \bT_k \bB_{k-1}^\T 
\hspace{2mm} \Rightarrow  \hspace{2mm}\bM_{k:l} =  \bB_k \bT_{k:l} \bB_{l}^\T, \label{eq:matrixdecomp}
\end{equation}
where we define $\bT_{k:l} \triangleq \bT_k \cdots \bT_{l}$.  We note that $\bM_{k:k}=\bI_n$ implies $\bT_{k:k} \equiv \bI_n$.  Let $\be_{i_j}$ denote the $i_j$-th standard basis vector, such that
\begin{align}
\l|\bM^{\T}_{k:l} \bB^{i_j}_k \r|^2  
= \be_{i_j}^\T \bT_{k:l} \(\bT_{k:l}\)^\T  \be_{i_j} 
=\l|\(\bT^\T_{k:l}\)^{i_j} \r|^2 
\end{align}
where $\(\bT^\T_{k:l}\)^{i_j}$ denotes the ${i_j}$-th column of $\bT^\T_{k:k-l}$, i.e., the ${i_j}$-th row of $\bT_{k:k-l}$.  For any $k$ and any $\epsilon>0$, there exists some $N_{\epsilon, k}$ such that if $k-l$ is taken sufficiently large, equation \cref{eq:backwardsgrowth} guarantees
\begin{align}
e^{2(\lambda_i - \epsilon)l} \leq \l|\(\bT^\T_{k:k-l}\)^{i_j} \r|^2 \leq e^{2(\lambda_i + \epsilon)l}.\label{eq:rownorm}
\end{align}

\begin{mydef}\label{def:lle}
For each $k>l$, $i=1,\cdots, p$ and $j=1, \cdots, \kappa_i$ we define the ${i_j}$-th \textbf{local Lyapunov exponent (LLE)} from $k$ to $l$ as $\frac{1}{k-l}\log \(\rvert T^{i_j}_{k:l} \rvert\)$ where $T^{i_j}_{k:l}$ is defined to be the ${i_j}$-th diagonal entry of $\bT_{k:l}$.
\end{mydef}
\begin{lemma}\label{lemma:computationalLE}
For any fixed $l$,
\begin{align}
\lim_{k\rightarrow \infty}\frac{1}{k-l}\log \(\rvert T^{i_j}_{k:l} \rvert\)&= \lambda_i
\end{align}
\end{lemma}
\begin{proof}
This is also discussed by Ershov \& Potapov \cite{ershov1998}, in demonstrating the convergence of the recursive QR algorithm.  For a discussion on the numerical stability and convergence see, e.g., Dieci \& Van Vleck \cite{dieci2002lyapunov, dieci2007lyapunov}.
\end{proof}

Perturbations of model error to the mean equation for the Kalman filter are not governed by the asymptotic rates of growth or decay, but rather, the LLEs.
While the LLE $\frac{1}{k-l}\log \(\rvert T^{i_j}_{k:l} \rvert\)$ approaches the value $\lambda_i$ as $k-l$ approaches infinity, 
its behavior on short time scales can be highly variable.  Particularly, for an arbitrary LP-regular system, the rate of convergence in equation \cref{eq:backwardsgrowth} may depend on $k$.  An important class of such systems is, e.g., non-uniformly hyperbolic systems \cite{barreira2002}[see chapter 2].  To make the LLEs tractable, we make an additional assumption, compatible with the typical assumptions for \emph{partial hyperbolicity} \cite{hasselblatt2006}.  We adapt the definition of partial hyperbolicity from Hasselblatt \& Pesin \cite{Hasselblatt2011} to our setting.

\begin{mydef}\label{def:partiallyhyperbolic}
Let $\lambda_{n_0}= 0$.  For every $k$ we define the splitting into unstable, neutral and stable subspaces: 
\begin{align}\label{eq:unneustab}
E^u_k \triangleq \bigoplus_{i= 1}^{n_0 - 1}\E^i_k, & & E^c_k \triangleq \E^{n_0}_k & &and & & E^s_k \triangleq \bigoplus_{i= n_0 + 1}^{p}\E^i_k.
\end{align}
Suppose there exists constants $C>0$ and 
\begin{align}
0 < \eta_s \leq \nu_s < \eta_c \leq \nu_c <  \eta_u \leq \nu_u
\end{align}
independent of $k$ such that $\nu_s <1 < \eta_u$ and for any $l>0$, $\bv \in E^m_k$, $\l| \bv \r|=1$, and $m\in\{s,c,u\}$
\begin{align}
\frac{(\eta_m)^l}{C} \leq \l| \bM_{k+l:k} \bv \r| \leq (\nu_m)^l C .\label{eq:parthypbound}
\end{align}
Then the model \cref{eq:dynmodel} is \textbf{(uniformly) partially hyperbolic (in the narrow sense)}.
\end{mydef}

Partially hyperbolic systems, as in \cref{def:partiallyhyperbolic}, have LLEs which are bounded uniformly with respect to rates defined on the subspaces in equation \cref{eq:unneustab}.  When $C$ is taken large the definition permits transient growth of stable modes and transient decay of unstable modes.  The neutral subspace encapsulates diverse behaviors which always fall below prescribed rates of exponential growth or decay.  We will make a slightly stronger assumption on these uniform growth and decay rates that is equivalent to fixing a uniform window of transient variability \emph{on each} Lyapunov exponent.
\begin{hyp}\label{hyp:uniformconvergence}
Let $\epsilon>0$ be given.  We assume that for each $i$ there exists some $N_{i,\epsilon}$, independent of $k$ and $j$, such that for any $\bB^{i_j}_k$ whenever $k-l> N_{i,\epsilon}$
\begin{align}\label{eq:epsilonwindow}
-\epsilon < \frac{1}{k-l} \log \(\l| \bM^{\T}_{k : l} \bB^{i_j}_k \r| \)  - \lambda_i < \epsilon ,
\end{align}
i.e., the growth and decay is uniform (translation invariant) in $k$.
\end{hyp}
Unless specifically stated otherwise, we assume \cref{hyp:uniformconvergence} for the remaining of this paper.  However, our results may be generalized to all systems satisfying \cref{def:partiallyhyperbolic} by using only the uniform rates of growth or decay on the \emph{entire} unstable, neutral and stable subspaces in equation \cref{eq:parthypbound}.  Our results also apply to systems without neutral exponents, i.e. $\lambda_{n_0}>0$, as a trivial extension.

\section{Dynamically induced bounds for the Riccati equation}
\label{section:bounds}
\subsection{Autonomous systems}\label{sec:aut}
Consider the classical theorem regarding the existence and uniqueness of solutions to the stable Riccati equation for autonomous dynamics.  This is paraphrased from Theorem 2.38, Chapter 7, of Kumar \& Varaiya \cite{kumar86} in terms of the forecast error covariance recurrence in equation \cref{eq:riccatitv}. 
\begin{mydef}
The autonomous system is defined such that for every $k$
\begin{equation}
\begin{matrix}
\bM_k \equiv \bM ,  &
\bH_k \equiv \bH ,  &
\bQ_k \equiv \bQ ,  &
\bR_k \equiv \bR  &
and &
\bOmega_k \equiv \bOmega.
\end{matrix}\label{eq:autdef}
\end{equation}
Let $\bP = \bX \bX^\T$ for some $\bX\in \mathbb{R}^{n\times n}$, the \textbf{stable Riccati equation} is defined as
      \begin{align}
      \bP &= \bM \bX(\bI_n + \bX^\T \bOmega \bX )^{-1} \bX^\T \bM^\T +\bQ \label{eq:stableRiccati}
      \end{align}
\end{mydef}
\begin{thrm}
Let the autonomous system defined by equations \cref{eq:dynmodel}, \cref{eq:obsmodel} and \cref{eq:autdef} satisfy \cref{hyp:obscont}.  There is a positive semi-definite matrix, $ \widehat{\bP}\equiv \widehat{\bX}\widehat{\bX}^\T$, which is the unique solution to the stable Riccati equation \cref{eq:stableRiccati}.  For any initial choice of $\bP_0$, if $\bP_k$ satisfies the recursion in equation \cref{eq:riccatitv}, then $\lim_{k\rightarrow \infty} \bP_k = \widehat{\bP}$.
\end{thrm}   

Slightly abusing notation, take $\alpha$ and $\beta$ to be defined by the solution to the stable Riccati equation \cref{eq:stableRiccati},
\begin{align}\label{eq:stablealphabeta}
&\alpha\triangleq \sigma_n\(\widehat{\bX}^\T \bOmega \widehat{\bX}\) \geq 0 & 
&\beta\triangleq \sigma_1\(\widehat{\bX}^\T \bOmega \widehat{\bX}\) < \infty.
\end{align}
Then for any $k$ we recover the invariant recursion for the stable limit
\begin{align}
(1+ \beta)^{-k}\bM^{k} \widehat{\bP}\(\bM^\T\)^{k} + \bXi^\beta_{k:1} \leq \widehat{\bP} \leq (1+ \alpha)^{-k}\bM^{k} \widehat{\bP}\( \bM^\T \)^{k} + \bXi^\alpha_{k:1}. \label{eq:stablerecursion}
\end{align}

\begin{prop}\label{prop:autupper}
Assume equations \cref{eq:dynmodel} and \cref{eq:obsmodel} satisfy \cref{hyp:obscont} and define $\alpha,\beta$ for the stable Riccati equation as in equation \cref{eq:stablealphabeta}.  For any $1\leq i\leq p$, if there exists $\epsilon > 0$ such that 
\begin{align}\label{eq:autepsilonupper}
\frac{e^{2(\lambda_i + \epsilon)}}{1+ \alpha} < 1,
\end{align}
choose $N_{i,\epsilon}$ as in \cref{hyp:uniformconvergence}. For the eigenvalue $\mu_i$ of $\bM^\T$, where $\rvert \mu_i \rvert = e^{\lambda_i}$, choose any eigenvector $\bv_{i_j}$.  Then  
\begin{align}
\bv_{i_j}^\T\widehat{\bP}\bv_{i_j} &\leq \frac{\bv_{i_j}^\T \bQ \bv_{i_j} }{1-\frac{e^{2\lambda_i} }{1+\alpha}}. \label{eq:eigbndupper}
\end{align}
  Moreover, if $\bB^{i_j}$ is the $i_j$-th BLV, then
\begin{align}
\(\bB^{i_j}\)^\T\widehat{\bP}\bB^{ i_j} &\leq  \(\bB^{ i_j}\)^\T\bXi^\alpha_{N_{i,\epsilon}:0}\bB^{ i_j} +  \(\frac{e^{2(\lambda_i +\epsilon)}}{1+ \alpha}\)^{N_{i,\epsilon} +1} \(\frac{q_{\sup}}{1-\frac{e^{2(\lambda_i+\epsilon)}}{1+\alpha}}\) \label{eq:coarseupperinv}.
\end{align}
 
For every $1\leq i\leq p$, any $\epsilon>0$, and associated $N_{i,\epsilon}$ as in \cref{hyp:uniformconvergence},
\begin{align}\label{eq:eigbndlower}
\frac{\bv_{i_j}^\T \bQ \bv_{i_j} }{1-\frac{e^{2\lambda_i} }{1+\beta}} \leq \bv_{i_j}^\T\widehat{\bP}\bv_{i_j}
\end{align}
and 
\begin{align}
\(\bB^{ i_j}\)^\T\bXi^\beta_{N_{i,\epsilon}:0}\bB^{ i_j} +  \(\frac{e^{2(\lambda_i - \epsilon)}}{1+ \beta}\)^{N_{i,\epsilon} +1} \(\frac{q_{\inf}}{1-\frac{e^{2(\lambda_i-\epsilon)}}{1+\beta}}\) \leq \(\bB^{ i_j}\)^\T\widehat{\bP}\bB^{ i_j} .\label{eq:coarselowerinv}
\end{align}
\end{prop}

\begin{proof}
Note that time invariant propagators trivially satisfy \cref{hyp:uniformconvergence} and it is easy to verify the relationship $\rvert \mu_i \rvert = e^{\lambda_i}$ directly from the definition of the Lyapunov exponents.  We begin by proving equations \cref{eq:eigbndupper} and \cref{eq:eigbndlower} for eigenvectors of $\bM^\T$.  If $\bv_{i_j}$ is an eigenvector of $\bM^\T$ associated to $\mu_i$, equation \cref{eq:stablerecursion} implies
\begin{align}
 \bv_{i_j}^\T\widehat{\bP}\bv_{i_j} \leq \(\frac{\rvert\mu_i\rvert^2 }{1+ \alpha}\)^{k+1}\bv_{i_j}^\T\widehat{\bP} \bv_{i_j} + \sum_{l=0}^{k} \(\frac{\rvert \mu_i \rvert^2 }{1+\alpha}\)^l \bv_{i_j}^\T\bQ \bv_{i_j}
\end{align}
for every $k$.  For $\lambda_i < 0$ generally, or for any $\lambda_i$ such that $\alpha > e^{2\lambda_i} - 1$,
\begin{align}
 &\lim_{k\rightarrow\infty}\[\(\frac{\rvert\mu_i\rvert^2 }{1+ \alpha}\)^{k+1}\bv_{i_j}^\T \widehat{\bP}\bv_{i_j} +\sum_{l=0}^{k} \(\frac{\rvert \mu_i \rvert^2 }{1+\alpha}\)^l \bv_{i_j}^\T \bQ \bv_{i_j}\]  = \frac{\bv_{i_j}^\T \bQ \bv_{i_j}}{1-\frac{\rvert \mu_i \rvert^2 }{1+\alpha}} 
\end{align} 
 and
\begin{align} 
 \bv_{i_j}^\T\widehat{\bP}\bv_{i_j}\leq \frac{\bv_{i_j}^\T \bQ \bv_{i_j}}{1-\frac{e^{2\lambda_i} }{1+\alpha}} .
\end{align}

The stable Riccati equation \cref{eq:stableRiccati} implies $\bQ \leq \widehat{\bP}$.  Therefore, using the left side of \cref{eq:stablerecursion} demonstrates that for \emph{any} eigenvector $\bv_{i_j}$
\begin{align}
 \sum^{k}_{l=0}   \(\frac{ \mid \mu_i \mid^{2}}{1+ \beta}\)^l \bv_{i_j}^\T \bQ \bv_{i_j} \leq \bv_{i_j}^\T\widehat{\bP}\bv_{i_j}  \label{eq:eigrecursionlower}
\end{align}
for all $k$.  In particular, for every eigenvector $\bv_{i_j}$ we obtain
\begin{align}\label{eq:eiglower}
\frac{\bv_{i_j}^\T \bQ \bv_{i_j}}{1 - \frac{\mid\mu_i\mid^2}{1 + \beta}} & \leq \bv_{i_j}^\T\widehat{\bP} \bv_{i_j} .
\end{align}
  
The above argument does not have a straightforward extension to the generalized eigenspaces so we coarsen the bound to obtain a closed limiting form in terms of the BLVs which retain the important growth characteristics under $\bM^\T$.  For $i> n_0$, or for any $\lambda_i$ such that $\alpha > e^{2\lambda_i} - 1$, there is a choice of $\epsilon$ as in equation \cref{eq:autepsilonupper} and $N_{i,\epsilon}$ as in \cref{hyp:uniformconvergence}.  Let $\widehat{\bP} \leq \widehat{p}_{\sup}\bI_n$, then from the right side of equation \cref{eq:stablerecursion} we derive
\begin{align}
\widehat{\bP}
 &\leq  \frac{ \widehat{p}_{\sup} \bM^{k+1} \(\bM^\T\)^{k+1}}{(1+ \alpha)^{k+1}}  + \sum_{l=0}^{k} \frac{\bM^{l} \bQ \(\bM^\T\)^{l}}{(1+\alpha)^l} \\
&\leq  \frac{ \widehat{p}_{\sup} \bM^{k+1} \(\bM^\T\)^{k+1}}{(1+ \alpha)^{k+1}}  + \bXi^\alpha_{N_{\epsilon,i}:1} + q_{\sup} \sum_{l=N_{\epsilon,i}+1}^{k} \frac{\bM^{l}  \(\bM^\T\)^{l}}{(1+\alpha)^l}, 
 \label{eq:reducedupperaut}
\end{align}
which implies $\(\bB^{ i_j}\)^\T\widehat{\bP}\bB^{ i_j}$ can be bounded above by
\begin{align}
 \widehat{p}_{\sup}\frac{ \l| \(\bM^\T\)^{k+1} \bB^{ i_j} \r|^2}{(1+ \alpha)^{k+1}} + \(\bB^{ i_j}\)^\T\bXi^\alpha_{N_{i,\epsilon}:1}\bB^{ i_j} + q_{\sup} \sum_{l=N_{i,\epsilon} + 1}^{k} \frac{ \l| \(\bM^\T\)^l \bB^{ i_j} \r|^2}{(1+\alpha)^{l}}.
\end{align}
Utilizing equation \cref{eq:backwardsgrowth} we bound $\(\bB^{ i_j}\)^\T\widehat{\bP}\bB^{ i_j}$ by
\begin{align}
\widehat{p}_{\sup}\(\frac{e^{2(\lambda_i +\epsilon)}}{1+ \alpha}\)^{k+1}  + \(\bB^{ i_j}\)^\T\bXi^\alpha_{N_{i,\epsilon}:1} \bB^{ i_j} + q_{\sup} \sum_{l=N_{i,\epsilon} +1}^{k} \(\frac{e^{2(\lambda_i+\epsilon)}}{1+\alpha}\)^{l} \label{eq:psiautbnd}
\end{align}
for every $k > N_{i,\epsilon}$.  Taking the limit of equation \cref{eq:psiautbnd} as $k\rightarrow \infty$ yields
\begin{align}
\(\bB^{ i_j}\)^\T\widehat{\bP}\bB^{ i_j} &\leq  \(\bB^{i_j}\)^\T\bXi^\alpha_{N_{i,\epsilon}:1}\bB^{ i_j} +   \(\frac{e^{2(\lambda_i +\epsilon)}}{1+ \alpha}\)^{N_{i,\epsilon} +1} \(\frac{q_{\sup}}{1-\frac{e^{2(\lambda_i+\epsilon)}}{1+\alpha}}\),
\end{align}
The lower bound is demonstrated by similar arguments with the lower bound in equation \cref{eq:stablerecursion}, utilizing the property $\widehat{\bP} < \infty$.
\end{proof}

\cref{prop:autupper} is similar results in perfect models \cite{bocquet2015expanding, gurumoorthy2017, bocquet2017degenerate}, but with some key differences.  Once again that the estimation errors are dissipated by the dynamics in the span of the stable BLVs, but the recurrent injection of model error prevents the total collapse of the covariance to the unstable-neutral subspace.  In equation \cref{eq:eigbndupper}, we see that for very strong decay, when $e^{2\lambda_i} \approx 0$, or high precision observations, i.e., when the system is fully observed and as $\alpha \rightarrow \infty$, the stable limit of the forecast uncertainty reduces to what is introduced by the recurrent injection of model error.  The SEEK filter of Pham et. al. \cite{pham1998} has exploited these properties by neglecting corrections in the stable eigenspaces and only making corrections in the unstable directions.  This is likewise the motivation for AUS of Trevisan et. al. \cite{carrassi2008b, trevisan2010, trevisan2011kalman, palatella2013a, palatella2015}, though the work of AUS was concerned with nonlinear, perfect models.  

The upper bounds in equations \cref{eq:eigbndupper} and \cref{eq:coarseupperinv} generally hold for $i\leq n_0$ only when the system is fully observed.  Therefore, these bounds can be considered an ideal bound for the unstable-neutral modes.  However, the lower bound in equation \cref{eq:eiglower} hold generally for $i<n_0$.  By assuming the existence of an invariant solution to the stable Riccati equation \cref{eq:stableRiccati}, we will recover a necessary condition for its existence.

\begin{cor}\label{cor:lowerbndspread}
\emph{Assume there exists a solution $\widehat{\bP}$ to the stable Riccati equation \cref{eq:stableRiccati}}.  Choose the smallest index $i$ such that $1\leq i \leq n_0$ and there exists some generalized eigenvector $\bv_{i_j}$ of $\bM^\T$ for which $\bv_{i_j} \notin {\rm \emph{null}}\(\bQ\)$.
\emph{Then it is necessary that}
\begin{align}\label{eq:autnecessary}
\frac{e^{2\lambda_i}}{1+\sigma_1^2\(\bR^{-\frac{1}{2}}\bH \widehat{\bX}\)} <1.
\end{align}
\end{cor}
\begin{proof}
Let $\bv_{i_1}$ be an eigenvector for $\bM^\T$ and $\bQ\bv_{i_1} \neq 0$.  Then by the definition of $\beta$ in equation \cref{eq:stablealphabeta}, the equation \cref{eq:eigrecursionlower} holds for all $k$ if and only if equation \cref{eq:autnecessary} holds.  More generally, suppose $\{\bv_{i_j}\}_{j=1}^{\kappa_i}$ are (possibly complex) generalized eigenvectors forming a Jordan block for $\bM^\T$.  Let $j$ be the smallest index for which $\bQ \bv_{i_j} \neq 0$.  Recall that for each $j\in\{1,\cdots,\kappa_i\}$ the Jordan basis satisfies
\begin{align}
\(\bM^\T -\mu_i \bI_n \)\bv_{i_j} = \bv_{i_{j-1}}
\end{align}
where $\bv_{i_0} \equiv 0$.  Therefore, for any $m\geq 1$, the vector $\(\bM^\T -\mu_i \bI_n \)^m \bv_{i_j}$ is in the span of $\{\bv_{i_1}, \cdots, \bv_{i_{j-1}}\}$.   Let us define $\bN \triangleq \bM^\T -\mu_i \bI_n$ so that
\begin{align}
\begin{split}
\sum_{l=0}^{k+1} \bQ \(\bM^\T\)^l \bv_{i_j} &= \sum_{l=0}^{k+1} \bQ\(\bN + \mu_i \bI_n\)^l  \bv_{i_j} \\
& = \sum_{l=0}^{k+1} \sum_{m=0}^l \mu_i^{l-m} {l \choose m} \bQ   \bN^m    \bv_{i_j} \\  
& = \sum_{l=0}^{k+1} \mu_i^l \bQ    \bv_{i_j}.
\end{split}\label{eq:jordanerror}
\end{align}
Multiply equation \cref{eq:stablerecursion} on the left with $\bv_{i_j}^\H$ (the conjugate transpose) and the right with $\bv_{i_j}$.  Combining this with the equality in equation \cref{eq:jordanerror}, proves the result. 
\end{proof}

\cref{cor:lowerbndspread} shows that it is necessary for the existence of the stable Riccati equation that observations are precise enough, relative to the background uncertainty, to counteract the strongest dynamic instability forcing the model error.  The quantity in \cref{eq:autnecessary} thus represents the stabilizing effect of the observations, similar to the bounds on the conditioning number provided by Haben et al. \cite{haben2011conditioning} and Tabeart et al. \cite{tabeart2018conditioning}, but in \cref{cor:lowerbndspread} expressly in response to the system's dynamic instabilities. 

\subsection{Time varying systems}
\label{section:tvbnd}
In the following, we will extend the results of \cref{prop:autupper} and \cref{cor:lowerbndspread} to time-varying systems, and derive a uniform bound on the unfiltered errors in the stable BLVs in \cref{cor:stablebnd}.

\begin{prop}\label{prop:tvbnd}
Assume equations \cref{eq:dynmodel} and \cref{eq:obsmodel} satisfy \cref{hyp:obscont} (b).  For any $1 \leq i \leq p$, if there exists $\epsilon > 0$ such that 
\begin{align}\label{eq:tvepsilonupper}
\frac{e^{2(\lambda_i + \epsilon)}}{1+ \alpha} < 1,
\end{align}
choose $N_{i,\epsilon}$ as in \cref{hyp:uniformconvergence}.  Then there exists a constant $0\leq C_{\alpha,N_{i,\epsilon}}$ such that
\begin{align}
\limsup_{k\rightarrow \infty} \(\bB^{ i_j}_k\)^\T\bP_k\bB^{ i_j}_k &\leq  C_{\alpha, N_{i\epsilon}} +  \(\frac{e^{2(\lambda_i +\epsilon)}}{1+ \alpha}\)^{N_{i,\epsilon} +1} \(\frac{q_{\sup}}{1-\frac{e^{2(\lambda_i+\epsilon)}}{1+\alpha}}\).\label{eq:tvupper}
\end{align}
 
If \cref{hyp:obscont} (a) is also satisfied, then for every $1 \leq i \leq p$, any $\epsilon>0$ and associated $N_{i,\epsilon}$, there exists $0\leq C_{\beta,N_{i,\epsilon}}$ such that
\begin{align}
  C_{\beta, N_{i\epsilon}}   +  \(\frac{e^{2(\lambda_i - \epsilon)}}{1+ \beta}\)^{N_{i,\epsilon} +1} \(\frac{q_{\inf}}{1-\frac{e^{2(\lambda_i-\epsilon)}}{1+\beta}}\) \leq \liminf_{k\rightarrow \infty} \(\bB^{ i_j}_k\)^\T\bP_k\bB^{ i_j}_k .\label{eq:tvlower}
\end{align}
\end{prop}
\begin{proof}
If the system satisfies \cref{hyp:obscont} (b) then
\begin{align}
\bXi^\alpha_{k:k-N_{i,\epsilon}} &\leq \bXi^0_{k:k-N_{i,\epsilon}} \equiv \bUpsilon_{k:k-N_{i,\epsilon}} \leq b_{N_{i,\epsilon}} \bI_n,
\end{align}
where $b_{N_{i,\epsilon}}$ is independent of $k$.  Therefore, there exists a constant depending on $\alpha$ and $N_{i,\epsilon}$, but independent of $k$, such that
\begin{align}
\bXi^\alpha_{k:k-N_{i,\epsilon}} \leq C_{\alpha,N_{i,\epsilon}} \bI_n.
\end{align} 
Let $\bP_0 \leq p_0 \bI_n $ bound the prior covariance.  Equation \cref{eq:nonautrecursion} implies
\begin{align}
 & \bP_k \leq p_0 \frac{\bM_{k:0} \bM^\T_{k:0} }{(1+\alpha)^k} +  q_{\sup}\sum_{l=1}^{k} \frac{\bM_{k:l} \bM^\T_{k:l}}{(1+\alpha)^{k-l}}.
 \end{align}
From the above, we bound $\(\bB^{ i_j}_k\)^\T\bP_k \bB^{i_j}_k$ with
 \begin{align}
 &p_0 \frac{\l| \bM^\T_{k:0} \bB^{ i_j}_k\r|^2}{(1+\alpha)^k} +  \(\bB^{ i_j}_k\)^\T\bXi^\alpha_{k:k-N_{i,\epsilon}}\bB^{ i_j}_k + q_{\sup}  \sum_{l=0}^{k-N_{i,\epsilon}-1} \frac{\l| \bM^\T_{k:l} \bB^{i_j}_k\r|^2}{(1+\alpha)^{k-l}},
 \end{align}
 thus
 \begin{align}
&\(\bB^{ i_j}_k\)^\T\bP_k \bB^{i_j}_k \leq p_0 \(\frac{e^{2(\lambda_i + \epsilon)}}{1+\alpha}\)^k +  C_{\alpha,N_{i,\epsilon}} + q_{\sup}  \sum_{l=N_{i,\epsilon}+1}^{k} \(\frac{e^{2(\lambda_i + \epsilon)}}{1+\alpha}\)^l \label{eq:backwardstvfinite}.
\end{align}
Taking the $\limsup$ in equation \cref{eq:backwardstvfinite} as $k \rightarrow \infty$ yields equation \cref{eq:tvupper}.  

Suppose that \cref{hyp:obscont} (a) and (b) are both satisfied, then by \cref{thrm:boundedstable} there exists a uniform bound on $\bP_k$ such that $\bX_k$ must also be uniformly bounded; together with uniform boundedness of $\bR_k$ and $\bH_k$, this implies
$\beta < \infty$.  Note that
\begin{align}
\bXi^\beta_{k:k-N_{i,\epsilon}} &\geq \(\frac{1}{1 + \beta}\)^{N_{i,\epsilon}} \bXi^0_{k:k-N_{i,\epsilon}} \geq \(\frac{1}{1 + \beta}\)^{N_{i,\epsilon}} a_{N_{i,\epsilon}} \bI_n  
\end{align}
for some constant $a_{N_{i,\epsilon}}$ independent of $k$.  This implies
\begin{align}
\bXi^\beta_{k:k-N_{i,\epsilon}} \geq C_{\beta,N_{i,\epsilon}} \bI_n \label{equation:controlbeta}
\end{align}
for a constant $C_{\beta,N_{i,\epsilon}}$ depending on $\beta$ and $N_{i,\epsilon}$ but independent of $k$.  Utilizing the recursion in equation \cref{eq:nonautrecursion}, choosing $\epsilon$ and an appropriate $N_{i,\epsilon}$, and finally bounding the weighted controllability matrix with equation \cref{equation:controlbeta} allows one to recover the lower bound in equation \cref{eq:tvlower} in a similar manner to the upper bound.
\end{proof} 

The above proposition shows that there is a uniform upper and lower bound on the forecast error for the Kalman filter, in the presence of model error, which can be described in terms of inverse, competing factors: the constant $\alpha$ (respectively $\beta$) is interpreted as the minimal (respectively maximal) observational precision relative to the maximal (respectively minimal) background forecast uncertainty, represented in the observation variables.
  Additionally $C_{\beta, N_{i\epsilon}}, C_{\alpha, N_{i\epsilon}}$ represent the lower and upper bounds on \emph{local variability} of the evolution of model errors, before perturbations adhere within an $\epsilon$ threshold to their asymptotic behavior.
 
\begin{cor}\label{cor:tvnecessary}
Assume equations \cref{eq:dynmodel} and \cref{eq:obsmodel} satisfy \cref{hyp:obscont} (b), and there exists uniform bound to the forecast error Riccati equation \cref{eq:riccatitv} for all $k$.  \emph{Then it is necessary that}
\begin{align}\label{eq:tvnecessary}
\frac{e^{2\lambda_1}}{1+\sup_k\sigma_1^2\(\bR_k^{-\frac{1}{2}}\bH_k \bX_k\)} <1.
\end{align}
\end{cor}
\begin{proof}
If the forecast error Riccati equation \cref{eq:riccatitv} is uniformly bounded, there is a $ 0<p_{\sup}<\infty$ such that we have the inequality, $\bP_{k} \leq p_{\sup} \bI_n$ for all $k$, and $\beta< \infty$.  Using the lower bound in equation \cref{eq:nonautrecursion}, for all $k$ we have
\begin{align}\label{eq:psupbound}
\( \frac{1}{1 + \beta} \)^k \bM_{k:0} \bP_0 \bM^\T_{k:0} + \bXi^\beta_{k:1}& \leq p_{\sup} \bI_n.
\end{align}
The summands in equation \cref{eq:psupbound} are positive semi-definite such that for any $k> N_\Upsilon+1$, truncating $\bXi_{k:1}^\beta$ verifies
\begin{align}
\sum_{l=1}^{N_\Upsilon + 1} \(\frac{1}{1 + \beta}\)^{k-l}\bM_{k:l}\bQ_l \bM_{k:l}^\T \leq \bXi_{k:1}^\beta \leq p_{\sup}\bI_n.
\end{align}
Note that by \cref{def:obscontmat}, if $k>N_{\Upsilon}+2$ 
\begin{align}
\bM_{k:N_\Upsilon + 1} \bUpsilon_{N_\Upsilon +1 :1} \bM_{k:N_\Upsilon + 1}^\T= \sum_{l=1}^{N_\Upsilon +1}\bM_{k:l}\bQ_l \bM_{k:l}^\T,
\end{align}
and therefore, for every $k>N_{\Upsilon}+2$
\begin{align}
  \(\frac{1}{1+\beta}\)^{k -1 } \bM_{k:N_\Upsilon + 1} \bUpsilon_{N_\Upsilon +1 :1} \bM_{k:N_\Upsilon + 1}^\T&\leq p_{\sup}\bI_n.
\end{align}
Using \cref{hyp:obscont} (b), for every $k>N_{\Upsilon}+2$ we derive
\begin{align}
 \(\frac{1}{1+\beta}\)^{k-N_\Upsilon - 1} \bM_{k:N_\Upsilon +1}\bM_{k:N_\Upsilon +1}^\T & \leq \frac{p_{\sup}( 1+\beta)^{N_\Upsilon}}{b}\bI_n,\label{eq:finalsupbound}
\end{align}
using the inequality in \cref{eq:controlbound}.  For any $j$, multiplying equation \cref{eq:finalsupbound} on the left by $\(\bB^{1_j}_k\)^\T$ and on the right by $\bB^{1_j}_k$ and taking the limit as $k\rightarrow \infty$ shows that it is necessary for equation \cref{eq:tvnecessary} to hold for the left side to be bounded away from $\infty$.
\end{proof}

In contrast to \cref{cor:lowerbndspread} for autonomous systems, \cref{cor:tvnecessary} uses the \cref{hyp:obscont} (b) to simplify the arguments --- this moreover guarantees the necessary criterion is with respect to $\lambda_1$, as the controllability matrix is guaranteed to be positive definite and thus nonvanishing on every Oseledec space.  There is, however, a more direct analogue to the statement of \cref{cor:lowerbndspread} where the \emph{adjoint-covariant Lyapunov vectors} will play the role of the eigenvectors of $\bM^\T$.  It is easy to demonstrate that the adjoint-covariant Lyapunov vectors have the desired covariance and growth/decay with respect to the reverse time adjoint model, $\bM^\T_k$.  There exist, under the condition of integrally separated Oseledec spaces, classical constructions for covariant and adjoint-covariant bases that decompose the model propagator into a block-upper-triangular form \cite{adrianova1995introduction}[see Theorem 5.4.9]. This decomposition makes the derivation of a precise statement like \cref{cor:lowerbndspread} analogous in time varying models, with respect to the adjoint-covariant Laypunov vectors and adjoint-covariant Oseledec spaces.  However, the above arguments require significant additional exposition which we feel unnecessary, as \cref{cor:tvnecessary} is sufficiently general.

\begin{cor}\label{cor:stablebnd}
Assume equations \cref{eq:dynmodel} and \cref{eq:obsmodel} satisfy \cref{hyp:obscont} (b) and suppose $\bH_k \bX_k \equiv \0$ for every $k$ such that $\alpha = \beta = 0$.  Let $k\geq 1$ and choose $\bv \in \emph{\sp}\left\{\bB^{i_j}_k : n_0 < i \leq p, \hspace{2mm} 1\leq j \leq \kappa_i\right\}$ such that $\l| \bv\r| =1 $.  There is a $C>0$ independent of $k$ such that
\begin{align}
\bv^\T\bP_k \bv & \leq C <\infty.
\end{align}
\end{cor}
\begin{proof}
The inequality in equation \cref{eq:nonautrecursion} is an equality for the unfiltered forecast where $\beta = \alpha=0$.  Thus the corollary is clear for any stable BLV directly from \cref{prop:tvbnd} and the conclusion extends to norm one linear combinations. 
\end{proof}

\cref{cor:stablebnd} extends the intuition of AUS to the presence of model error: corrections may be targeted along the expanding modes while the uncertainty in the stable modes remains bounded by the system's dynamic stability alone.  Particularly, \emph{without filtering} uncertainty remains uniformly bounded in the span of the stable BLVs. This is analogous to the results of Bocquet et. al. \cite{bocquet2017degenerate}, where in perfect models, the stable dynamics alone are sufficient to dissipate forecast error in the span of the stable BLVs.  With $\alpha=0$, the uniform bound in \cref{cor:stablebnd} may be understood by the two components which equation \cref{eq:tvupper} is composed of, the bound on $\bUpsilon_{k:k-N_{i,\epsilon}}$ and 
\begin{align}\label{eq:digestederror}
\frac{q_{\sup} e^{2(\lambda_i +\epsilon)N_{i,\epsilon} +1}}{1-e^{2(\lambda_i+\epsilon)}}.
\end{align}
The controllability matrix $\bUpsilon_{k:k-N_{i,\epsilon}}$ represents the newly introduced uncertainty from model error that \emph{is yet to be dominated} by the dynamics.  Equation \cref{eq:digestederror} represents an upper bound on the past model errors that have already been dissipated by the stable dynamics.  Nevertheless, this uniform bound is uninformative about the local variability.  In the following sections, we study the variance of the unfiltered uncertainty in the stable BLVs compared to the uncertainty of the Kalman filter.

\section{Numerical experiments}\label{section:numerics}
\subsection{Experimental setup}

To satisfy \cref{hyp:lpregular}, we construct a discrete, linear model from the nonlinear Lorenz-96 \textbf{(L96)} equations \cite{lorenz96}, commonly used in data assimilation literature see, e.g., \cite{carrassi2018data}[and references therein].
For each $m\in\{1,\cdots,n\}$, the L96 equations read $\frac{{\rm d}\bx}{{\rm d} t} \triangleq \bL(\bx)$,
\begin{align}\label{eq:l96}
 L^m(\bx) &=-x^{m-2}x^{m-1} + x^{m-1}x^{m+1} - x^m + F
\end{align}  
such that the components of the vector $\bx$ are given by the variables $x^m$ with periodic boundary conditions, $x^0=x^n$, $x^{-1}=x^{n-1}$ and $x^{n+1}=x^{1}$.  The term $F$ in L96 is the forcing parameter.  The tangent-linear model \cite{kalnay2003atmospheric}  is governed by the equations of the Jacobian matrix, $\nabla \bL(\bx)$,
\begin{align}\label{eq:jacobian}
 \nabla L^m(\bx) = \(0, \cdots, -x^{m-1}, x^{m+1} - x^{m-2}, -1, x^{m-1},0, \cdots, 0\).
\end{align}     

We fix the model dimension $n\triangleq 10$ and the forcing parameter as the standard $F=8$, as the model exhibits chaotic behavior, while the small model dimension makes the robust computation of Lyapunov vectors numerically efficient.  The linear propagator $\bM_k$ is generated by computing the discrete, tangent-linear model \cite{kalnay2003atmospheric} from the resolvent of the Jacobian equation \cref{eq:jacobian} along a trajectory of the L96, with an interval of discretization at $\delta \triangleq 0.1$.  We numerically integrate the Jacobian equation with a fourth order Runge-Kutta scheme with a fixed time step of $h \triangleq 0.01$.

For $F=8$, the 10 dimensional nonlinear L96 model has a non-degenerate Lyapunov spectrum and we replace the superscript $i_j$ with $i$ for the BLVs.  The model has three positive, one neutral and six negative Lyapunov exponents, such that $n_0 = 4$.  The Lyapunov spectrum for the discrete, linear model is computed directly via the relationship in \cref{lemma:computationalLE}, where the average is taken over $10^5$ iterations of the recursive QR algorithm, after pre-computing the BLVs to convergence.  In our simulations, before our analysis, we pre-compute the BLVs and the FLVs over $10^5$ iterations of the recursive QR algorithm for the forward model, or respectively, for the reverse time adjoint model \cite{Kuptsov2012}[see section 3].  We note that the computed Lyapunov spectrum for the discrete, linear model as in simulations is related to the spectrum of the nonlinear L96 model by rescaling the linear model's exponents by $\frac{1}{\delta}$.

\subsection{Variability of recurrent perturbations}
\label{section:localvariability}

While \cref{cor:stablebnd} guarantees that the uncertainty in the stable BLVs is uniformly bounded, this bound strongly reflects the scale of the model error and the local variance of the Lyapunov exponents.  If model errors are large, or the stable Lyapunov exponents have high variance, this indicates that the uniform bound can be impractically large for forecasting.  
Assume no observational or filtering constraint, i.e., $\bH_k\bX_k\equiv \0$.  Suppose that the model error statistics are uniform in time and spatially uncorrelated with respect to a basis of BLVs: $\bQ_k \triangleq  \bB_k \bD \bB_k^\T$, where $\bD\in \mathbb{R}^{n\times n}$ is a fixed diagonal matrix with the $i_j$-th diagonal entry given by $D_{i_j}$.  Denote $\bP_0 \equiv \bQ_0$, then equation \cref{eq:nonautrecursion} becomes
\begin{align}\begin{split}
 \(\bB^{ i_j}_k\)^\T \bP_k \bB^{i_j}_k & = \sum_{l=0}^k \(\bB^{ i_j}_k\)^\T \bM_{k:l} \bQ_l \bM^\T_{k:l}  \bB^{ i_j}_k \\
& =  \sum_{l=0}^k \(\be_{i_j}\)^\T \bT_{k:l} \bD \bT^\T_{k:l}  \be_{i_j}
\\
& = D_{i_j}\sum_{l=0}^k \l| \(\bT^\T_{k:l}\)^{i_j} \r|^2 ,
\end{split}
\label{eq:freeforecast}
\end{align}
where $\l| \(\bT^\T_{k:l}\)^{i_j} \r|$ is the norm of the $i_j$-th row of $\bT_{k:l}$.  In equation \cref{eq:freeforecast}, $\bP_k$ represents the freely evolved uncertainty at time $k$, and thus $\sum_{l=0}^k \l| \(\bT^\T_{k:l}\)^{i_j} \r|^2$ describes the variance of the free evolution of perturbations in the direction of $\bB^{i_j}_k$.
\begin{mydef}
For each $1 \leq i \leq p$, each $1 \leq j \leq \kappa_i$ and any $k$, we define
\begin{align}\label{eq:invscaling}
\Psi_k^{i_j} \triangleq \sum_{l=0}^k \l| \(\bT^\T_{k:l}\)^{i_j} \r|^2
\end{align}
to be the \textbf{free evolution of perturbations} in the direction of $\bB^{i_j}_k$.
\end{mydef}

Assuming the errors are uncorrelated in the basis of BLVs is a strict assumption, but studying the free evolution of perturbations has general applicability: $\bB_k^\T \bQ_k \bB_k \leq  q_{\sup}\bI_n$, and therefore, equation \cref{eq:invscaling} may be interpreted in terms of an upper bound on the variance of the freely evolved forecast uncertainty in the $i_j$-th mode.  \cref{alg:bound} describes our recursive approximation of the free evolution, given by equation \cref{eq:invscaling}, for $k\in \{1,\cdots,m\}$ via the QR algorithm.  We assume that the QR algorithm has been run to numerical convergence for the BLVs at time $0$.

\begin{algorithm}[H]
\caption{Free evolution of perturbations in the $i_j$-th BLV}
\label{alg:bound}
\begin{algorithmic}
\STATE{Define $\bB_0$ to be the BLVs at time zero and $m\geq 1$.}
\FOR{$k= 1, \cdots, m $}
\STATE{Let $\bT_k, \bB_k$ be defined by the QR recursion \cref{eq:qrrecursion}, and let $\bT^s_{k}\in \mathbb{R}^{s\times s}$ be the lower right sub-matrix of $\bT_k$ corresponding to the stable exponents.}
\STATE{Set $\Psi^{i_j}_k = 1$ and $\bT \triangleq \bI_s$.}
        \FOR{$l = 0, \cdots, k-1$}
            \STATE{$\bT \coloneqq \bT\times\bT^s_{k-l}$.}
            \STATE{$\Psi^{i_j}_k \coloneqq \Psi^{i_j}_k + \l| \( \bT^\T\)^{i_j} \r|^2$ for each $i=n_0 +1,\cdots, p$ and $j= 1,\cdots , \kappa_i$.}
        \ENDFOR
        \RETURN $\Psi^{i_j}_k$
\ENDFOR
\end{algorithmic}
\end{algorithm}

\begin{rmk}
Equation \cref{eq:rownorm} implies $\l| \( \bT^\T\)^{i_j} \r|^2$ decays exponentially in $k-l$ and the inner loop of \cref{alg:bound} needs only be computed to the first $l$ such that $\l| \( \bT^\T\)^{i_j} \r|^2$ is numerically zero.
\end{rmk}

The approximation of \cref{eq:invscaling} with \cref{alg:bound} is numerically stable for all $k$ and any $i > n_0$, precisely due to the upper triangular dynamics in the BLVs.  The upper triangularity of all $\bT_k$ means the lower right block of $\bT_{k:l}$ is given as the product of the lower right blocks of the sequence of matrices $\{\bT_j\}_{j= l+1}^k$.  Therefore, computing the stable block of $\bT_{k:l}$ is independent of the unstable exponents, and the row norms of $\bT_{k:l}$ converge uniformly and exponentially to zero by \cref{hyp:uniformconvergence}.

In \cref{fig:bnd_v_lle} we plot $\Psi_k^{5}$ and $\Psi_k^6$ as in \cref{alg:bound} and the LLEs for $\bB^5_k$ and $\bB^6_k$ for $k\in\{1, \cdots, 10^4\}$.  Assuming that $\bQ_k \leq  q_{\sup}\bI_n$, $\Psi^i_k$ bounds the variance in the $i$-th stable mode at time $k$, up to the scaling factor of $q_{\rm sup}$.  As $n_0=4$, the exponent $\lambda_5$ is the stable exponent closest to zero.  The left side of \cref{fig:bnd_v_lle} corresponds to the exponent $\lambda_5 \approx -0.0433$ while the right side corresponds to the exponent $\lambda_6 \approx -0.0878$.  The upper row in \cref{fig:bnd_v_lle} plots the evolution of $\Psi^i_k$ for $\bB^5_k$ and $\bB^6_k$, while  the bottom row shows the corresponding time series of LLEs.  The mean of the LLEs are approximately equal to their corresponding Lyapunov exponent, while the standard deviation is given by $0.142$ for $\lambda_5$ and $0.133$ for $\lambda_6$ respectively.

\begin{figure}[h]\label{fig:bnd_v_lle}
\center
\includegraphics[width=\linewidth]{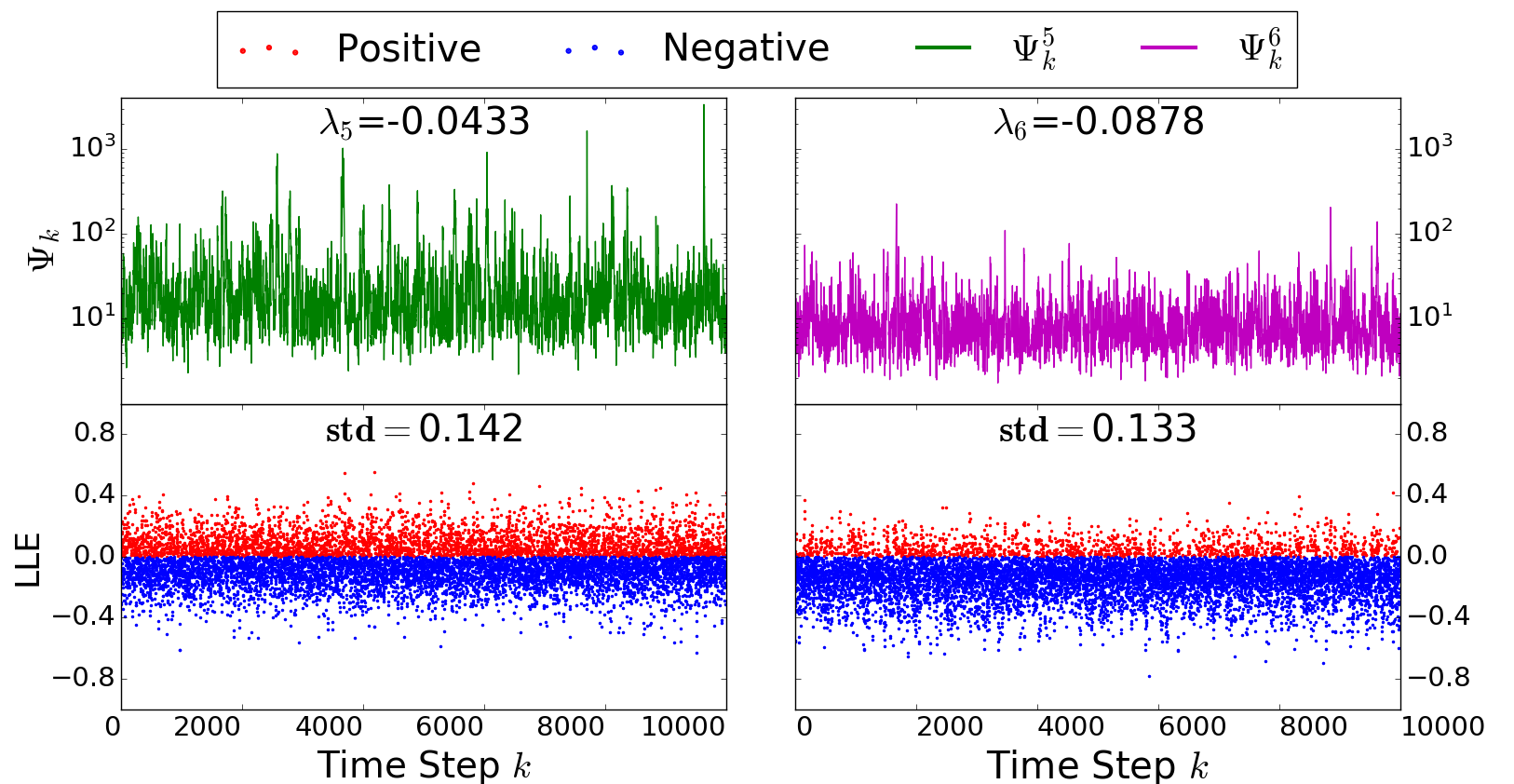}
\caption{\textbf{Horizontal axis:} time step $k \in\{1,\cdots 10^4\}$ \textbf{Upper row:} time series of $\Psi^{5}_k$ and $\Psi^{6}_k$.  \textbf{Lower row:} local Lyapunov exponents of the fifth and sixth backward vector. \textbf{Left column:} $\lambda_5=-0.0433$. \textbf{Right column:} $\lambda_6=-0.0878$.}
\end{figure}

While $\Psi_k^5$ is uniformly bounded, \cref{fig:bnd_v_lle} illustrates that it can be on the order of $\mathcal{O}\(10^3\)$, with a mean value of approximately 808 over the $10^4$ iterations.  This is in contrast to perfect models where the projection of the unfiltered forecast error into any stable mode converges to zero at an exponential rate \cite{bocquet2017degenerate}.  Moreover, the frequency and scale of positive realizations of LLEs of $\bB^5_k$ has a strong impact the variance of the unfiltered error.  The fewer, and weaker, positive realizations of the LLEs of $\bB^6_k$ correspond to the lower overall uncertainty represented by $\Psi^6_k$.  The maximum of $\Psi_k^6$ is on the order of $\O(10^{2})$, with a mean value of approximately 28.

\subsection{Unfiltered versus filtered uncertainty}
\label{section:kferr}

In the following, we compare the variance of the unfiltered error in the stable BLVs, represented by $\Psi_k^i$, for $i\in\{5,\cdots,10\}$, with the uncertainty in the Kalman filter.  Assuming that $\bQ_k \triangleq \bI_n$, in this case $\Psi_k^i$ is equal to the variance of the unfiltered error along $\bB^i_k$.  While the error in the Kalman filter depends on the observational configuration, the value of $\Psi_k^i$ depends only on the underlying dynamics.  Therefore, we benchmark the variance of the unfiltered error over a range of observational designs to determine under what conditions the unfiltered error in the stable BLVs will exceed the uncertainty of the full rank Kalman filter.  This analysis allows us to evaluate how many of the stable BLVs can remain unfiltered while achieving an acceptable forecast performance.  This comparison has a special significance when considering reduced rank, sub-optimal filters, which is the subject of the sequel \cite{grudzien2018inflation}.

The recent works of Bocquet et al. \cite{bocquet2017degenerate} and Frank \& Zhuk \cite{frank2017detectability}, weaken \cref{hyp:obscont} to criteria on the observability, or detectability, of the unstable-neutral subspace to obtain filter stability and boundedness in perfect models.  The results in \cref{cor:lowerbndspread}, \cref{cor:tvnecessary} and \cref{cor:stablebnd} similarly suggest that the sufficient condition for filter boundedness, \cref{hyp:obscont}, may be weakened in the presence of model errors. For this reason, we will study the variance of the filtered error with respect to observations satisfying the criteria discussed by Bocquet et al. \cite{bocquet2017degenerate} and Frank \& Zhuk \cite{frank2017detectability}.

Given a fixed dimension of the observational space $d < n$, consider finding an observational operator, $\bH_k$, which minimizes the forecast uncertainty.  Suppose the singular value decomposition of an arbitrary choice of $\bH_k$ is given as
\begin{align}
\bH_k = \bU_k \bSigma_k \bV_k^\T.
\end{align}
For a given observation error covariance matrix, the size of the singular values of $\bH_k$ correspond to the precision of observations relative to the uncertainty in the precision matrix, $\bOmega_k \triangleq \bH_k^{\T}\bR^{-1}_k\bH_k$.
Imposing that all singular values of $\bH_k$ must be equal to one, then up to an orthogonal transformation of $\bR^{-1}_k$, we equate the choice of an observational operator $\bH_k$ with the selection of an orthogonal matrix $\bV_k \in \mathbb{R}^{n \times d}$.

For perfect models, $\bQ_k \equiv 0$, we write the forecast error Riccati equation in terms of a choice of $\bH_k \triangleq \bV_k^\T$ as
\begin{align}
\bP_{k+1} &= \bM_{k+1} \bX_{k} \[\bI_n + \(\bR_k^{-\frac{1}{2}} \bV^\T_k \bX_{k}\)^\T \(\bR_k^{-\frac{1}{2}}\bV_k^\T \bX_{k}\)\]^{-1}\bX^\T_{k}\bM_{k+1}^\T.
\end{align}
The Frobenius norm, $\l| \bP_{k+1} \r|_\bF = \sqrt{\tr\( \bP_{k+1}^2\)} 
,$ is bounded by
\begin{align} 
 &\tr\left\{  \bM_{k+1} \bX_{k}\[\bI_n + \(\bR_k^{-\frac{1}{2}} \bV^\T_k \bX_{k}\)^\T \(\bR_k^{-\frac{1}{2}}\bV_k^\T \bX_{k}\)\]^{-1}\bX^\T_{k}\bM_{k+1}^\T\right\} \\
 \leq& \tr\left\{  \[\bI_n + \(\bR_k^{-\frac{1}{2}} \bV^\T_k \bX_{k}\)^\T \(\bR_k^{-\frac{1}{2}}\bV_k^\T \bX_{k}\)\]^{-1}\right\} \tr\( \bX^\T_{k}\bM_{k+1}^\T \bM_{k+1} \bX_{k} \).\label{eq:frobineq}
\end{align}
Equation \cref{eq:frobineq} attains its smallest values when the eigenvalues of
\begin{align}\label{eq:eigs}
\bI_n + \(\bR_k^{-\frac{1}{2}} \bV^\T_k \bX_{k}\)^\T \(\bR_k^{-\frac{1}{2}}\bV_k^\T \bX_{k}\)
\end{align}
are as large as possible, similar to maximizing the denominator of equation \cref{eq:tvnecessary}.

For a fixed sequence of observation error covariances, finding the largest eigenvalues of equation \cref{eq:eigs} can be studied by finding the subspace for which the matrix of orthogonal projection coefficients $\bV^\T_{k} \bX_{k}$ has the largest singular values.  In perfect models, the forecast error covariance for the Kalman filter asymptotically has support confined to the span of the unstable and neutral BLVs \cite{gurumoorthy2017, bocquet2017degenerate}.  This is likewise, evidenced for the ensemble Kalman filter in weakly-nonlinear models \cite{ng2011role,bocquet2017four}, suggesting that the columns of $\bV_k$ should be taken as the leading $d$ BLVs. 
\begin{mydef}\label{mydef:unstableobs}
Given $d \geq 1$, let $\bB^{1:d}_k \in \mathbb{R}^{n\times d}$ denote the matrix comprised of the first $d$ columns of $\bB_k$.  We define the observation operator $\bH^{\rm bd}_k \triangleq \(\bB_k^{1:d}\)^\T$.
\end{mydef}

\cref{mydef:unstableobs} is a formalization of the AUS observational paradigms \cite{trevisan2004assimilation, carrassi2007adaptive} utilizing ``bred vectors'' as proxies for the BLVs.  The breeding method of Toth \& Kalnay \cite{toth1997ensemble} simulates how the modes of fast growing error are maintained and propagated through the successive use of short range forecasts in weather prediction.  The bred vectors are formed by initializing small perturbations of a control trajectory and forecasting these in parallel along the control. Upon iteration, the span of these perturbations generically converge to the leading BLVs.  For a discussion of variants of this algorithm, and the convergence to the BLVs, see e.g., Balci et al.  \cite{balci2012ensemble}.

The choice of observation operator in \cref{mydef:unstableobs} is also related to the numerical study of targeted observations for the L96 model of Law et al. \cite{law2016filter}.  Law et al. target observations with the eigenvectors of the operator $\bM^\T_{k+1} \bM_{k+1}$, but note that for a small interval $\delta \triangleq t_{k+1} - t_k$, the difference between the linearized equations defining $\bM^\T_{k+1} \bM_{k+1}$ and $\bM_{k+1}\bM^\T_{k+1}$ is negligible \cite{law2016filter}[see Remark 5.1].  Law et al. suggest that the eigenvectors of either  $\bM^\T_{k+1} \bM_{k+1}$ and $\bM_{k+1}\bM^\T_{k+1}$ may be sensible depending on whether the filter should take into account the principle axes of growth from the past to the current time or from the present to future time.  It is clear from equations \cref{eq:farfuture} and \cref{eq:adjointfarpast} that as $\delta$ becomes large, the eigenvectors of $\bM_{k+1}\bM^\T_{k+1}$ approach the BLVs, whereas $\bM^\T_{k+1} \bM_{k+1}$ approach the FLVs.

\begin{mydef}\label{mydef:Fobs}
Given $d \geq 1$, let $\bF^{1:d}_k \in \mathbb{R}^{n\times d}$ denote the matrix comprised of the first $d$ columns of $\bF_k$.  We define the observation operator $\bH^{\rm fd}_k \triangleq \(\bF_k^{1:d}\)^\T$.
\end{mydef}

Note that the observation operator $\bH_k^{\rm b4}$ uniformly-completely observes the span of the unstable and neutral BLVs, and thus for $d\geq 4$, $\bH_k^{\rm bd}$ satisfies the sufficient criterion for filter stability in perfect dynamics discussed by Bocquet et al. \cite{bocquet2017degenerate}.  The operator $\bH_k^{\rm b4}$ likewise satisfies the necessary and sufficient detectability criterion for filter stability perfect dynamics of Frank \& Zhuk \cite{frank2017detectability}.  On the other hand, the operator $\bH^{\rm fd}_k$ observes the span of the leading $d$ FLVs.  Unlike the BLVs, the FLVs define a QL decomposition of the span of the covariant Lyapunov vectors \cite{Kuptsov2012}[see equation (53)].  This implies that the columns of the operator $\bF_k^{\rm f4}$ actually spans the orthogonal complement to the stable Oseledec spaces.  Therefore, $\bH_k^{\rm f4}$ satisfies the criterion of Frank \& Zhuk \cite{frank2017detectability}, but will not generally satisfy the condition of Bocquet et al. \cite{bocquet2017degenerate}.

We perform parallel experiments, fixing the sequence of linear propagators $\bM_k$, and the initial prior error covariance $\bP_0 \triangleq \bI_n$, while varying the choice of the observation operator and the observational dimension $d$.  In each parallel experiment, we study the average forecast uncertainty for the full rank Kalman filter as described by Frobenius norm of the forecast error covariance $\bP_k$, averaged over $10^5$ assimilations, neglecting a separate filter stabilization period of $10^4$ assimilations.  For each $d\in \{ 4, \cdots, 9\}$, we compare the following choices of observation operators: (i) $\bH_k^{\rm bd}$; (ii) $\bH_k^{\rm fd}$; (iii) $\bH_k \triangleq \bV_k^\T$ for randomly drawn orthogonal matrices, $\bV_k \in \mathbb{R}^{n \times d}$; and (iv) a fixed network of observations, given by the leading $d$ rows of the identity matrix, i.e., $\bH_k \triangleq \bI_{d\times n}$.  We also compute the average Frobenius norm of the forecast error covariance for full dimensional observations, with $\bH_k \triangleq \bI_n$.  In each experiment, the observational and model error covariances are fixed as $\bR_k \triangleq \bI_d$ and $\bQ_k \triangleq \bI_n$.  For each $i$, the value of $\Psi^i_k$ is averaged over the $10^5$ assimilations.

In \cref{fig:obs_comp_b}, we plot the average Frobenius norm of the Kalman filter forecast error covariance matrix as a function of the number of observations, $d$.  We consider the observation configurations $\bH^{\rm bd}_k$, $\bV^\T_k$ and $\bI_n$ (plotted horizontally).  The average values of $\Psi_k^i$ for $i=7, \cdots,10$ are also plotted horizontally.  While the observational dimension $d < 7$, the average uncertainty for the Kalman filter with random observations, or observations in the BLVs, is greater than the average variance of the unfiltered error along $\bB^7_k$.  Similarly, in \cref{fig:obs_comp_f} we consider the configurations with observations defined by $\bH_k^{\rm bd}$, $\bH^{\rm fd}_k$ and $\bI_{d \times n}$.  The average values of $\Psi_k^i$ for $i=5, \cdots,8$ are plotted horizontally.  The variance of the unfiltered error in $\bB^5_k$ exceeds the uncertainty of the Kalman filter in every configuration.  The Kalman filter with observations fixed, or in the FLVs, do not obtain comparable performance with the unfiltered error in $\bB^6_k$ until $d\geq 6$.  Only the LLEs of $\bB^i_k$ for $i=8,9,10$ are sufficiently stable to bound the unfiltered errors below the Kalman filter with a fully observed system.

\begin{figure}[h]\label{fig:obs_comp_b}
\center
\includegraphics[width=.75\linewidth]{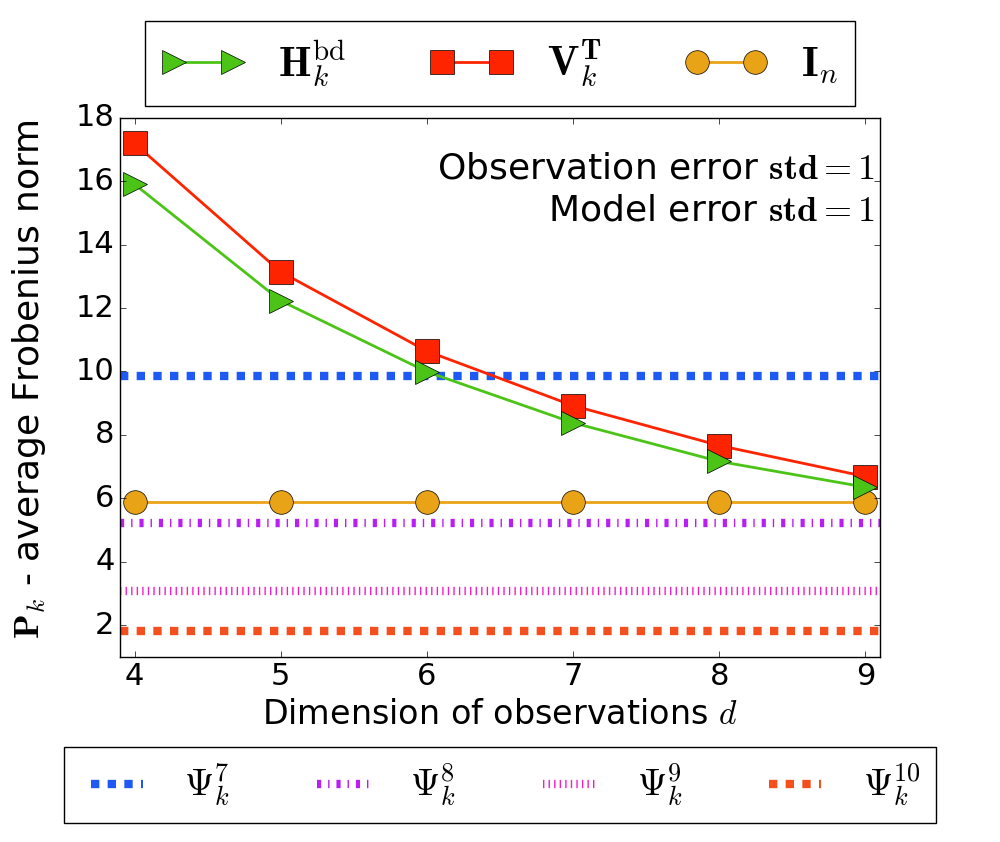}
\caption{Average Frobenius norm of $\bP_k$ over $10^5$ observation-forecast cycles, with dynamic \emph{($\bH^{\rm bd}_k$)} and random \emph{($\bV_k^\T$)} observations plotted versus the observational dimension $d$.  Average of variance in $i$-th BLV, $\Psi^i_k$, $i=7,\cdots, 10$ over $10^5$ observation-forecast cycles plotted horizontally.  Average Frobenius norm of $\bP_k$ with full dimensional observations, \emph{($\bI_n$)}, plotted horizontally.}
\end{figure}

\begin{figure}[h]\label{fig:obs_comp_f}
\center
\includegraphics[width=.75\linewidth]{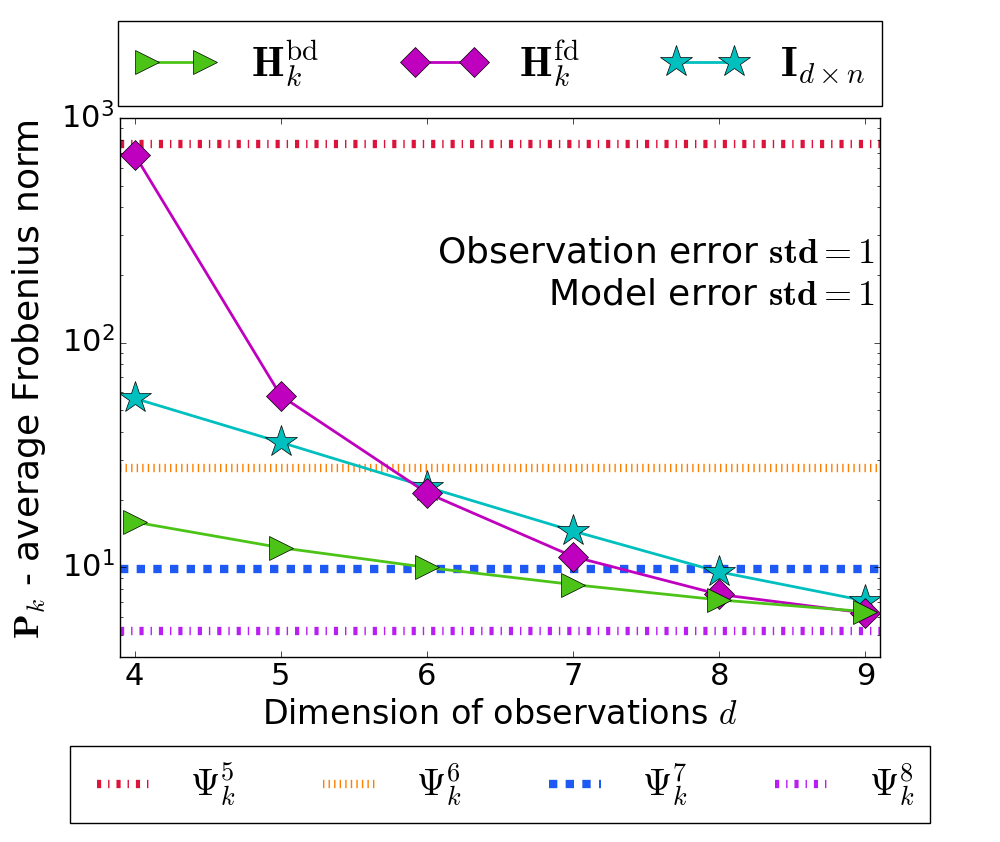}
\caption{Average Frobenius norm of $\bP_k$ over $10^5$ observation-forecast cycles, with BLV \emph{($\bH^{\rm bd}_k$)}, FLV \emph{($\bH_k^{\rm fd}$)}, and fixed observations \emph{($\bI_{d\times n}$)} plotted versus the observational dimension $d$.  Average of variance in $i$-th BLV, $\Psi^i_k$, $i=5,\cdots, 8$ over $10^5$ observation-forecast cycles plotted horizontally.}
\end{figure}

Our results have strong implications for the necessary rank of the gain in ensemble-based Kalman filters.  In perfect, weakly nonlinear models, the ensemble span typically aligns with the leading BLVs \cite{ng2011role,bocquet2017four}.  From the above results, we conclude that the effective rank of the ensemble-based gain must be increased to account for weakly stable BLVs of high variance in the presence of model errors.  The perturbations of model errors excited by transient instabilities in these modes can lead to the unfiltered errors becoming unacceptably large compared to the filtered errors.  

In \cref{fig:obs_comp_b} and \cref{fig:obs_comp_f}, the choice of observations in the span of the leading BLVs dramatically outperforms the observations in the span of the leading FLVs, or fixed observations.  Likewise $\bH^{\rm bd}_k$ makes a slight reduction to the overall forecast error over a choice of $d$ random observations.  As the span of the leading $n_0$ FLVs is orthogonal to the trailing, stable Oseledec spaces, this choice is can be considered closer to the \emph{minimum necessary} observational constraint on the forecast errors.  Particularly, the kernel of $\bF^{{\rm f} n_0}_k$ is identically equal to the sum of the stable Oseledec spaces.  This suggests that a necessary and sufficient condition for filtered boundedness can be described in terms of the observability of the $n_0$ leading FLVs, similar to the criterion of Frank \& Zhuk \cite{frank2017detectability}.  While it is not \emph{necessary}, the sufficient condition of Bocquet et al. \cite{bocquet2017degenerate} leads to a lower filter uncertainty as the span of the leading $n_0$ BLVs generally contains the largest projection of the forecast error.  This suggest that observing the leading eigenvectors of $\bM_{k+1}\bM^\T_{k+1}$ may generally outperform observing the leading eigenvectors of $\bM^\T_{k+1} \bM_{k+1}$ when the time between observations $\delta = t_{k+1} -t_k$ leads to significant differences in the linear expansions, as was noted as an alternative design by Law et al. \cite{law2016filter}.  For operational forecasting, this supports the use of the breeding technique \cite{toth1997ensemble} to target observations, over using the axes of forward growth.

\section{Conclusion}\label{section:conclusion}

This work formalizes the relationship between the Kalman filter uncertainty and the underlying model dynamics, so far understood in perfect models, now in the presence of model error.  Generically, model error prevents the collapse of the covariance to the unstable-neutral subspace and our \cref{prop:autupper} and \cref{prop:tvbnd} characterize the asymptotic window of uncertainty.  We provide a necessary condition for the boundedness of the Kalman filter forecast errors for autonomous and time varying dynamics in \cref{cor:lowerbndspread} and \cref{cor:tvnecessary}: the observational precision, relative to the background uncertainty, must be greater than the leading instability which forces the model error.  Particularly, \cref{cor:stablebnd} proves that forecast errors in the span of the stable BLVs remain uniformly bounded, in the absence of filtering, by the effect of dynamic dissipation alone.  

The uniform bound on the errors in the span of the stable BLVs extends the intuition of AUS to the presence of model error, but with qualifications.  Studying this uniform bound with \cref{alg:bound}, we identify an important mechanism for the growth of forecast uncertainty in sub-optimal filters: variability in the LLEs for asymptotically stable modes can produce transient instabilities, amplifying perturbations of model error.  The impact of stable modes close to zero differs from similar results for nonlinear, perfect models by Ng. et. al. \cite{ng2011role}, and Bocquet et. al. \cite{bocquet2015expanding}, where the authors demonstrate the need to correct weakly stable modes in the ensemble Kalman filter due to the sampling error induced by nonlinearity.  Likewise, this differs from the EKF-AUS-NL of Palatella \& Trevisan \cite{palatella2015}, that accounts for the truncation errors in the estimate of the forecast uncertainty in perfect, nonlinear models.  Our work instead establishes the fundamental impact of these transient instabilities as a linear effect in the presence of model errors.

In addition to our necessary criterion for filter boundedness, in \cref{section:kferr} we discuss the criteria of Bocquet et al. \cite{bocquet2017degenerate} and Frank \& Zhuk \cite{frank2017detectability} in relation to dynamically targeted observations.  Our numerical results suggest how these sufficient, and respectively necessary and sufficient, criteria can be extended to the presence of model errors.   Moreover, we distinguish between the minimal necessary observational constraints for filter boundedness and more operationally effective, sufficient designs.  Particularly, our results suggest that while it may be necessary that the observations uniformly completely observe the span of the unstable-neutral FLVs, it is sufficient and improves performance to uniformly completely observe the span of unstable-neutral BLVs.  In terms of operational forecasting, this strongly supports the use of bred vectors to target observations to constrain the forecast errors.

  \cref{cor:lowerbndspread}, \cref{cor:tvnecessary}, \cref{cor:stablebnd} and the results of \cref{section:numerics} suggest that as a theoretical framework for the ensemble Kalman filter, AUS may be extended to the presence of model errors.  By uniformly completely observing and correcting for the growth of uncertainty in the span of the unstable, neutral and some number of stable BLVs, reduced rank filters in the presence of model errors may obtain satisfactory performance. In practice, one may compute off-line the typical uncertainty in the stable BLVs via \cref{alg:bound} and determine the necessary observational and ensemble dimension at which unfiltered forecast error has negligible impact on predictions.  However, computational limits on ensemble sizes may make this strategy unattainable in practice --- the impact of these unfiltered errors on the performance of a reduced rank, sub-optimal filter is the subject of the direct sequel to this work \cite{grudzien2018inflation}.

\section*{Acknowledgments}
The authors thank two anonymous referees, and their colleagues Karthik Gurumoorthy, Amit Apte, Erik Van Vleck, Sergiy Zhuk and Nancy Nichols, for their valuable feedback and discussions on this work.  CEREA is a member of the Institut Pierre-Simon Laplace (IPSL).
 
\bibliographystyle{plain}
\bibliography{references}
\end{document}